\documentclass[11pt,reqno]{amsart}

\usepackage[usenames,dvipsnames]{xcolor}
\usepackage{amsmath,amsthm,verbatim,amssymb,amsfonts,amscd, graphicx, enumerate, enumitem,dsfont}
\usepackage[hidelinks]{hyperref}
\usepackage{graphics}
\usepackage{mathtools}
\usepackage[font=small]{caption}
\usepackage{mathrsfs}
\usepackage{parskip}
\usepackage[numbers]{natbib}
\setcitestyle{open={[},close={]}}

\newcommand{\nocontentsline}[3]{}
\newcommand{\tocless}[2]{\bgroup\let\addcontentsline=\nocontentsline#1{#2}\egroup}

\usepackage[margin=1.5in,marginparwidth=60pt]{geometry}
\usepackage{amssymb}
\usepackage{xcolor}
\usepackage{graphicx}
\usepackage{times}

\definecolor{MutedBlue}{RGB}{40,90,160}

\newcommand{\Z}{\mathbb{Z}}
\newcommand{\R}{\mathbb{R}}
\newcommand{\Q}{\mathbb{Q}}
\def\C{\mathbb{C}}

\newcommand{\CP}{\mathbb{C}P}

\newcommand{\rank}{\operatorname{rank}}

\let\int\relax
\newcommand{\int}{\mathring}

\DeclareMathOperator{\lk}{{lk}}

\DeclareMathOperator{\writhe}{writhe}
\DeclareMathOperator{\rot}{rot}
\DeclareMathOperator{\tb}{tb}

\newtheoremstyle{thm}{2.5 pt}{2.5 pt}{\itshape}{}{\bfseries}{.}{.5em}{}
\theoremstyle{thm}
\newtheorem{theorem}{Theorem}[section]
\newtheorem*{theorem*}{Theorem} 
\newtheorem{corollary}[theorem]{Corollary}

\newtheorem{proposition}[theorem]{Proposition}

\newtheorem*{question*}{Question} 
 
\newtheorem*{thom*}{Theorem \thetheorem~(Thom Conjecture)}
\newtheorem{problem}[theorem]{Problem}
\newtheorem*{mingenusprob*}{The Minimal Genus Problem}

\newtheoremstyle{def}{2.5 pt}{2.5 pt}{}{}{\bfseries}{.}{.5em}{}
\theoremstyle{def}
\newtheorem{definition}[theorem]{Definition}
\newtheorem{remark}[theorem]{Remark}

\newtheorem{example}[theorem]{Example}
\newtheorem{exercise}[theorem]{Exercise}

\begin{document}




	\noindent \textbf{\huge A knot-theoretic tour of dimension four}  
	
	\smallskip

		\textbf{\large M\'arton Beke and Kyle Hayden}

\vspace{-1ex}
		
		\emph{\large Singularities and Low Dimensional Topology Winter School\\ January 23-27, 2023} 
	
	


	\setlength{\parskip}{0pt}
	
	\bigskip
	
	{\small \tableofcontents}
	
\section*{\textbf{Introduction}}	
	
	
		\setlength{\parskip}{.5\baselineskip plus 2pt}

These notes  follow a lecture series at the \emph{Singularities and low dimensional topology} winter school at the  R\'enyi Institute in January 2023, with a target audience of graduate students in singularity theory and low-dimensional topology. The lectures  discuss the basics of four-dimensional manifold topology, connecting this rich subject to knot theory on one side and to contact, symplectic, and complex geometry (through Stein surfaces) on the other side of the spectrum. For a more complete treatment of the topics discussed, we refer the reader to the beautiful books by Gompf and Stipsicz  \cite{gompfstipsicz} and Saveliev \cite{saveliev}. For the complete collection of lecture notes from the winter school, see \cite{singularities:book}.

	We will begin by investigating the basic anatomy and algebraic topology of 4-manifolds. The first lecture will focus on using handle diagrams to determine algebro-topological data, and the second lecture will focus on the intersection form of a 4-manifold.  As motivation for these first two lectures, our goal will be to assemble the ingredients of the following:

	\begin{theorem*}[Freedman \cite{freedman}]\label{thm:non-smoothable}
		There exists a closed, simply connected, topological 4-manifold that does not admit a smooth structure.
	\end{theorem*}

	The third lecture will focus on studying surfaces in 4-manifolds and the minimal genus problem, which asks for the minimal genus of an embedded surface representing a fixed second homology class. This will culminate in a sketch of the following landmark theorem:

	\begin{theorem*}[Donaldson \cite{donaldson}, Freedman \cite{freedman}]\label{thm:exoticR4}
		There exist smooth 4-manifolds that are homeomorphic but not diffeomorphic to $\R^4$.
	\end{theorem*}

	The fourth and final lecture will offer a brief look at contact structures on 3-manifolds and Stein structures on 4-manifolds. Constraints on smoothly embedded surfaces (in the form of an adjunction inequality) will provide a key tool for the minimal genus problem.

	\subsubsection*{Conventions} Unless otherwise stated, homology is taken with $\Z$ coefficients. We will typically assume our manifolds are connected, except when considering disconnected $n$-manifolds as boundaries of connected $(n+1)$-manifolds (e.g., in the context of cobordism). The symbol $\cong$ will be used to denote whatever notion of equivalence is most natural in a given context (e.g., diffeomorphism for smooth manifolds, homeomorphism for topological manifolds, isomorphism for groups or lattices, etc).
	
	\bigskip
	
	\medskip
	
	\section*{\textbf{Lecture 1: Handles, homotopy, and homology}}
	\addtocounter{section}{1}
	
	\medskip

	
	
{1.0.\hspace{1pt} \textbf{The goal.}} 	This lecture will focus on describing 4-manifolds and understanding their basic algebraic topology in terms of handle decompositions. As motivation for the first two lectures, our goal will be to assemble the ingredients of the following:
	
	\begin{theorem}[Freedman \cite{freedman}]\label{thm:non-smoothable}
		There exists a closed, simply connected, topological 4-manifold that does not admit a smooth structure.
	\end{theorem}

	\subsection{Handles, homotopy, and homology}\label{subsec:homology}
	Let $X$ be a smooth 4-manifold equipped with a handle decomposition.  Recall that an \textit{$n$-dimensional $k$-handle} is a copy of $D^k \times D^{n-k}$, which may be attached to the boundary of an $n$-manifold $X$ by fixing an embedding $\partial D^k \times D^{n-k} \hookrightarrow \partial X$. A \textit{handle decomposition} of $X$ is a decomposition of $X$ into a union of $n$-manifolds $X_0 =\emptyset  \subset X_1 \subset X_2 \subset \cdots$ such that $X_i$ is obtained from $X_{i-1}$ by a handle attachment.

	\begin{remark}[Anatomy of a handle]
		The \textit{core} of a $k$-handle $h=D^k\times D^{n-k}$ is the central $D^k\times 0$, its boundary is  the \textit{attaching sphere} $S^{k-1}\times 0$, and $S^{k-1}\times D^{n-k}$ is the \textit{attaching region}.
		Dually, we call $0 \times D^{n-k}$ the \textit{cocore}, its boundary $0 \times S^{n-k-1}$ the \textit{belt sphere}, and the corresponding component $D^k \times S^{n-k-1}$ of $\partial h$ is called the \textit{belt region}.
	\end{remark}

	For this lecture, let $X$ be a smooth, compact 4-manifold equipped with a (finite) handle decomposition. As a matter of  convenience, we will assume that $X$ is connected, hence admits a handle decomposition that has a unique 0-handle. Moreover, we will assume that handles are attached in order of increasing index and that all handles of a given index are attached simultaneously. We will be focusing on 4-manifolds built without 3-handles, but we note that if a closed 4-manifold contains a single 4-handle,  then the 3-handles are essentially uniquely determined by the handles of index $\leq 2$  \cite{LP72}. 
	
	Our first goal is to use the handle decomposition to read off the algebraic topology of $X$.  The key insight is that the handle decomposition ``collapses'' onto a cell decomposition for a topological space homotopy equivalent to the manifold $X$. Roughly speaking, each $k$-handle $D^k \times D^{4-k}$ collapses onto the $k$-cell $D^k \times 0$ (i.e., the handle's core disk), dragging all later cell attachments along by homotopy.

	\begin{example}\label{ex:wedge} 
		If $X$ consists only of one 0-handle and $m$ handles of index $k=1$ or $k=2$, then $X$ is homotopy equivalent to the $m$-fold wedge of $k$-spheres $\vee^m S^k$. (Collapse the 0-handle to a point, and all attaching regions $\partial D^k \times D^{4-k}$ collapse to that point, too.) Note that each $S^k$ can be viewed as the union of the core $k$-disk of a $k$-handle and a (singular) $k$-disk formed as the cone of the attaching sphere inside $B^4= \operatorname{Cone}(S^3)$;  this is depicted schematically for $m=1$ and $k=2$ in Figure~\ref{fig:trace-with-cone}. Alternatively, it will often be useful to replace this singular $k$-disk with a smooth $k$-manifold in $B^4$ bounded by the attaching sphere in $S^3$.

		\begin{figure}[!t]
			\center
			\def\svgwidth{.95\linewidth} 
			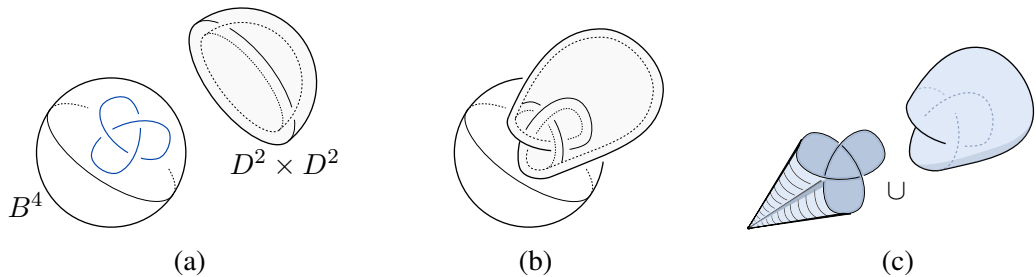
			
			\medskip
			
			\caption{Exhibiting a singular  2-sphere inside a 4-manifold built from one 0-handle and one 2-handle, known as a \emph{knot trace} (cf. Example~\ref{ex:wedge}).}\label{fig:trace-with-cone}
		\end{figure}
	\end{example}

	\begin{exercise}\label{exer:trace}
		The 4-manifold obtained by attaching an $n$-framed 2-handle to $B^4$ along a knot $K$ is commonly called the \emph{$n$-trace of $K$} and denoted $X_n(K)$. Using handle slides and cancellations, show that the 4-manifold depicted on the left side of Figure~\ref{fig:not942} (which uses dotted circle notation for  1-handles, as in \cite[\S5.4]{gompfstipsicz}) can be expressed as a 0-framed knot trace in two different ways.  \emph{(One of these knots is known as $9_{42}$ in the knot tables, while the other is identified as \texttt{o9\_{34801}} in  SnapPy  \cite{snappy}.)}
	\end{exercise}

	\begin{figure}[h]\center
		\def\svgwidth{.8\linewidth} 
		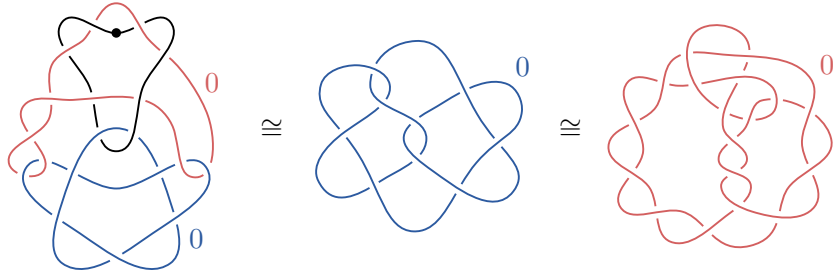 
		\caption{Kirby diagrams for a 4-manifold that can be expressed as the 0-trace of two different knots.}\label{fig:not942}
	\end{figure}
	
	Next, we consider a collection of 4-manifolds that will serve as running examples for our discussion.

	\begin{figure}
		\center
		\def\svgwidth{.675\linewidth} 
		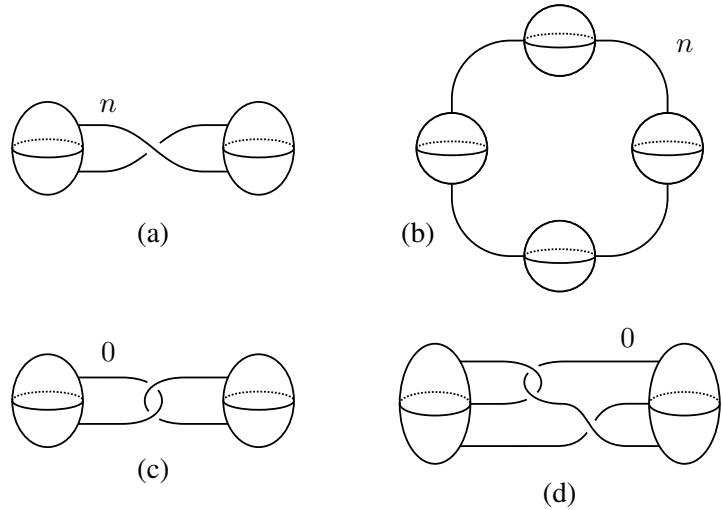
		\caption{Handle diagrams for the 4-manifolds in Examples~\ref{ex:running-ex}, \ref{ex:1-handle-examples}, and \ref{ex:H2}.} \label{fig:1-handle-examples}
	\end{figure}

	\begin{example}\label{ex:running-ex}
		Consider the 4-manifolds represented by the Kirby diagrams in Figure~\ref{fig:1-handle-examples}: {(a)} the $D^2$-bundle over $\mathbb{R}P^2$ of Euler number $n$; {(b)} the $D^2$-bundle over $T^2$ of Euler number $n$; {(c)} a self-plumbing of $S^2 \times D^2$, which may be viewed as regular neighborhood of an immersed 2-sphere with square zero in a 4-manifold; {(d)} a contractible 4-manifold known as a \textit{Mazur manifold}.
	\end{example}
	
	Now let $X$ be any smooth, connected 4-manifold with a handle decomposition that has a unique 0-handle. Moreover, assume that handles are attached in order of increasing index and that all handles of a given index are attached simultaneously.

	In general, $\pi_1(X)$ is determined by the union of handles of index $\leq 2$. After choosing orientations for the handles' attaching spheres, we can extract a presentation of $\pi_1(X)$ as follows.
	
	\begin{itemize}
		\item  \textit{Generators $\approx$ 1-handles:} Each generator is represented by an oriented \textit{core} curve that passes once over the given 1-handle (Figure~\ref{fig:1-handle-generator}).

		\item \textit{Relations $\approx$ 2-handles:} Each relation is determined by a 2-handle by expressing its oriented \textit{attaching} curve (up to homotopy) as a word in the 1-handles' core curves (i.e., the order and sign with which 1-handles are encountered as one traverses the attaching curve). 
	\end{itemize}
	
	\begin{figure}
		\center
		\includegraphics[width=.35\linewidth]{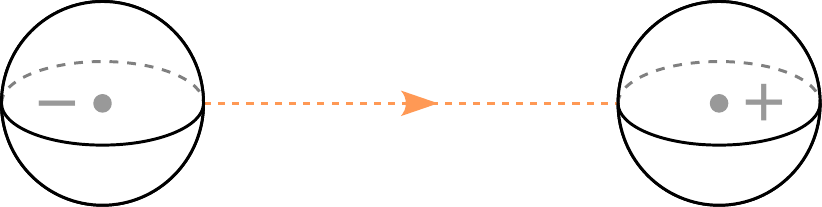}
		\caption{The two spheres are identified, and the other part of the circle travels through the core of the handle which we don't picture.}\label{fig:1-handle-generator}
	\end{figure}

	We note that  changing the attaching sphere's orientations will affect the presentation of the group but not its isomorphism type. Also note that the relations are independent of the framings of the 2-handles.

	\begin{example}\label{ex:1-handle-examples} For the 4-manifolds in Figure~\ref{fig:1-handle-examples}: {(a)} $\pi_1 \cong \langle x \mid x^2 \rangle \cong \Z/2\Z$ \ {(b)} $\pi_1 \cong \langle x,y \mid xyx^{-1} y^{-1}\rangle \cong \Z \oplus \Z$ \ {(c)} $\pi_1 \cong \langle x \mid x x^{-1} \rangle \cong \Z$ \ {(d)} $\pi_1 \cong \langle x \mid  x x^{-1} x \rangle \cong 1$
	\end{example}
	
	When calculating $\pi_1(X)$ from a handle diagram, it can be convenient to modify the 2-handles' attaching curves up to homotopy. For example, Figure~\ref{fig:homotopy} uses a crossing change to show that the attaching curve for the 2-handle in Figure~\ref{fig:1-handle-examples}(c) is nullhomotopic, hence contributes a trivial relation.
	
	\begin{figure}[b]
		\center
		\includegraphics[width=\linewidth]{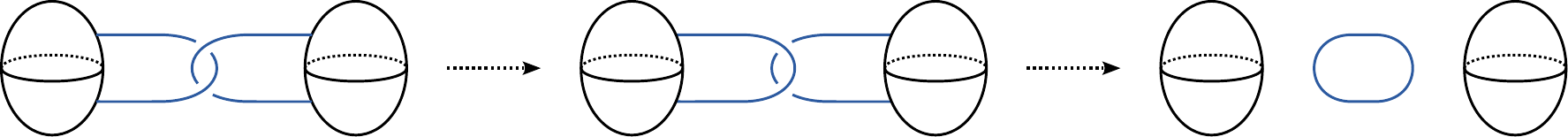}
		\caption{Modifying an attaching curve by a homotopy that involves a crossing change.}\label{fig:homotopy}
	\end{figure}

	\begin{exercise}  
		\begin{enumerate}
			\item [\normalfont\textbf{(a)}]  Use handle diagrams to build a 4-manifold $X$ such that $\pi_1(X)$ is nonabelian. 
			
			\item[\normalfont\textbf{(b)}] Prove that every finitely presented group arises as $\pi_1(X)$ for a smooth, compact, orientable 4-manifold $X$.
			
			\item [\normalfont\textbf{(c)}] Upgrade (b)  to achieve a \textit{closed} 4-manifold. (Hint: Don't start over.)
			
		\end{enumerate}
		
	\end{exercise}
	
	To calculate $H_1(X)$, we simply abelianize $\pi_1(X)$. The homology groups $H_k(X)$ with $k\geq 2$ can be calculated using cellular homology: the $k^\text{th}$ chain group $C_k(X)$ is generated by $k$-handles, and the boundary formula for a $k$-handle $h$  is given by
	$$\partial_k h = \sum (B_i \cdot A) \, h_i$$
	where the sum ranges over the $(k-1)$-handles $h_i$ and $A$ denotes the attaching sphere of the $k$-handle $h$, $B_i$ denotes the belt sphere of $h_i$, and $B_i \cdot A$ denotes  their signed intersection.
	
	\begin{example}\label{ex:H2} For the 4-manifolds in parts (a) and (d) of Figure~\ref{fig:1-handle-examples}, the groups $H_2(X)$ are trivial. In parts (b) and (c), we have $H_2(X) \cong \Z$. Note that, because these examples have no 3- or 4-handles,  the chain groups $C_3(X)$ and $C_4(X)$ are trivial; as a consequence, we have $H_2(X) \cong \ker \partial_2  \subset C_2(X)$. 
	\end{example}
	
	\begin{remark}
		We have now shown that the manifold $X$ in Figure~\ref{fig:1-handle-examples}(d) has $\pi_1(X)=1$ and $H_k(X)=1$ for all $k \geq 1$, hence applying  the Hurewicz and Whitehead  theorems implies that $X$ is contractible.
	\end{remark}

	\subsection{The intersection form} The most important algebro-topological invariant of a 4-manifold $X$ is its  cohomology ring, especially the multiplicative structure on second cohomology. 

	\begin{definition}
		Given a compact, orientable 4-manifold $X$, fix an orientation on $X$ by choosing a fundamental class $[X]$ generating $H_4(X,\partial X)\cong \Z$. The \textit{intersection form}
		of $X$ is the symmetric bilinear form 
		$$Q_X : H^2(X,\partial X) \otimes H^2(X,\partial X) \to \Z \qquad \quad (a,b) \mapsto \langle a \smile b, [X]\rangle. $$
	\end{definition}
	
	\begin{remark}
		Note that the cup product is symmetric on $H^2(X,\partial X;\Z)$. Analogous forms can be defined on $H^n$ for any $2n$-dimensional manifold; the form will be symmetric if $n$ is even and skew-symmetric if $n$ is odd.  
	\end{remark}
	
	The next result shows that intersection forms completely determine the homeomorphism type of closed, simply connected 4-manifolds.

	\begin{theorem}[Freedman \cite{freedman}]
		A pair of smooth, closed, simply connected 4-manifolds $X$ and $X'$ are homeomorphic $\iff$ $Q_{X_1}$ and $Q_{X_2}$ are isomorphic. 
	\end{theorem}

	How do we calculate intersection forms? Note that, by Poincar\'e-Lefschetz duality, we have $H^2(X,\partial X) \cong H_2(X)$, so we can view this as a bilinear form on homology
	$$Q_X : H_2(X) \otimes H_2(X) \to \Z \qquad \quad (\alpha,\beta) \mapsto \langle PD(\alpha) \smile PD(\beta), [X]\rangle. $$
	When $X$ is a smooth manifold, we can represent classes in $H_2(X)$ by smoothly embedded, oriented surfaces (Exercise~\ref{ex:represent}) and recast $Q_X$ as counting signed intersections  between them (cf. Proposition~\ref{prop:linking}). We will implement this using a diagrammatic approach below.
	
	\subsubsection{Linking numbers}
	Let $K_1$ and $K_2$ be disjoint, oriented knots in $S^3$. A basic exercise shows that any knot complement in $S^3$ has $H_1 \cong \Z$ and is generated by the meridian of $K_1$, leading to the following definition. 
	
	\begin{definition}
		Given oriented knots $K_1, K_2\subset S^3$, their \textit{linking number} $\lk(K_1,K_2)$ is the integer corresponding to the class  $[K_2] \in H_1(S^3 \setminus K_1) \cong \Z$, where $1 \in \Z$ corresponds to a positively oriented meridian of $K_1$.
	\end{definition}

	\begin{remark} The linking number can be expressed in other ways:
		\begin{enumerate}
			\item[(a)]  It can be computed diagrammatically via the formula
			$$\lk(K_1,K_2)= \frac{1}{2} \sum \{\text{signed crossings between $K_1$ and $K_2$}\}$$
			where the signs are determined via the right-hand rule as below. This makes it clear that the linking number is symmetric: $\lk(K_1,K_2)=\lk(K_2,K_1)$. 
			\begin{figure}[!ht]
				\center
				\includegraphics[width=.3\linewidth]{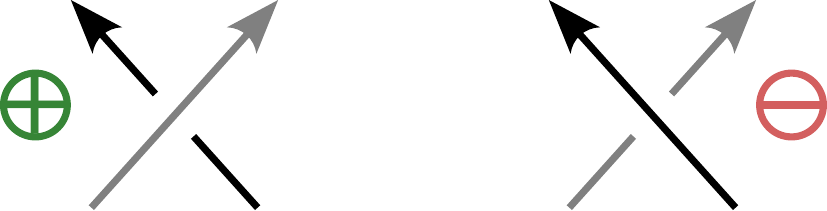}
			\end{figure}
			
			\item[(b)] The class $[K_2] \in H_1(S^3 \setminus K_1)$ can be calculated by choosing an oriented Seifert surface $\Sigma_1$ for the oriented knot $K_1$, perturbing $K_2$ to be transverse to $\Sigma_1$, then counting the signed intersection number $K_2 \cdot \Sigma_1$.
		\end{enumerate}
	\end{remark}
	
	\begin{example}\label{ex:linking}
		Consider the oriented link $L= K_1 \cup K_2 \cup K_3$ shown in Figure~\ref{fig:linking-ex}. A direct calculation using either of the above methods shows
		$$\lk(K_1,K_2)=2 \qquad \quad \lk(K_1,K_3) =0 \qquad \quad \lk(K_2,K_3)=0.$$
		In particular, note that  all four crossings between $K_1,\ K_2$ are positive, so $\lk(K_1,K_2)=2$. The component $K_3$ is unknotted and bounds an obvious disk in the page; it has two oppositely-signed intersection points with $K_1$, so $\lk(K_1,K_3)=0$. Finally, we also have $\lk(K_2,K_3)=0$ because $K_2$ and $K_3$ have no crossings.
	\end{example}
	
	\begin{figure}[h]
		\center
		\includegraphics[width=.85\linewidth]{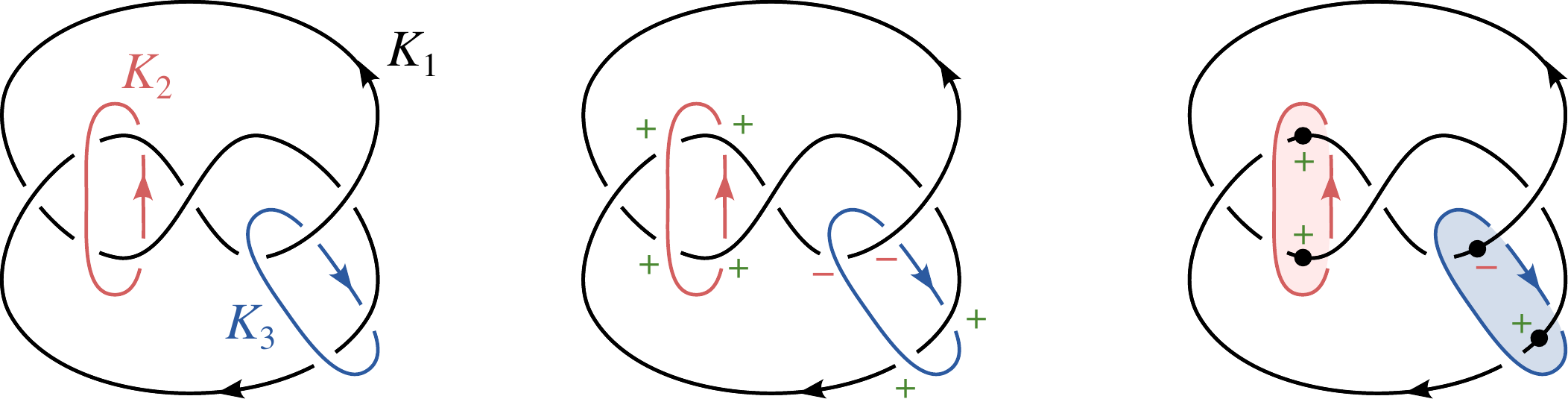}
		\caption{Calculating linking numbers for Example~\ref{ex:linking}.}\label{fig:linking-ex}
	\end{figure}

	Framings on link components can be interpreted as self-linking numbers, enabling us to organize the linking data of a framed, oriented link into a matrix as follows.

	\begin{definition}
		Given a framed, oriented $n$-component link $L \subset S^3$ (e.g., attaching data for 2-handles) with components $K_i$, its \emph{linking matrix} is the symmetric $n \times n$ matrix $A$ whose $ij^\text{th}$ entry is
		$$a_{ij} = \begin{cases} \lk(K_i,K_j) & i \neq j \\ \operatorname{framing}(K_i) & i = j. \end{cases}$$
	\end{definition}
	
	\begin{remark}
		A more uniform definition could be obtained by using the framings to define parallel pushoffs $K_i'$ for each $K_i$, then  setting $a_{ij}=\lk(K_i,K_j')$ for all $i$ and $j$.
	\end{remark}

	\begin{figure}[b]
	\center
		\includegraphics[width=.21\linewidth]{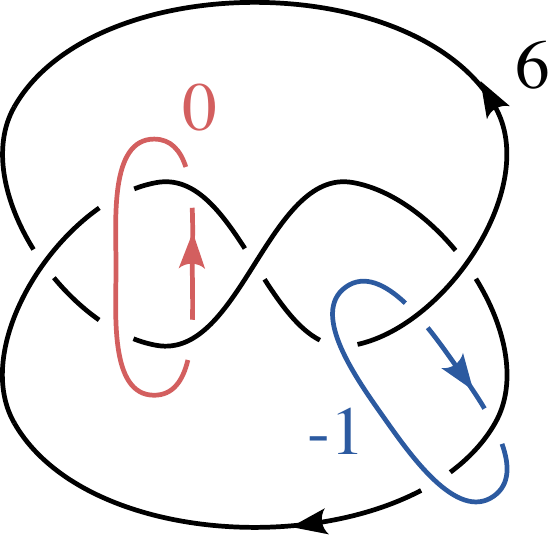}
		\caption{A framed, oriented link.} \label{fig:framingex}
	\end{figure}

	\begin{example}\label{ex:framed}
		The diagram in  Figure~\ref{fig:framingex} shows a copy of the link from Example~\ref{ex:linking} equipped with framings on the link components. With the orientations and notation from Figure~\ref{fig:linking-ex}, the linking matrix is
		$$\left[\ \begin{matrix}  6 & 2 & 0 \\ 2 & 0 & 0 \\ 0 & 0 & -1  \end{matrix}\ \right].$$

	\end{example}

	\subsubsection{Calculating intersection forms from handle diagrams} 
	
	The next proposition explains our interest in the linking matrix:

	\begin{proposition}[Linking matrices and intersection forms]\label{prop:linking}
		If $X$ is obtained from $B^4$ by attaching 2-handles along a framed, oriented $n$-component link $L$, then $H_2(X)\cong \Z^n$ and $Q_X$ is presented by the linking matrix of $L$.
	\end{proposition}
	
	\begin{proof}[Sketch] As mentioned in Example~\ref{ex:wedge}, $X$ is homotopy equivalent to $\vee^n S^2$, implying  $H_2(X) \cong \Z^n$. The intersection form $Q_X$ is determined by its behavior on any chosen basis for $H_2(X)$, and we can choose a natural one as follows.  For each component $K_i$ in $L$, fix an oriented Seifert surface $\Sigma_i \subset S^3$ for $K_i$. (For later convenience, we will assume $\Sigma_i$ is in general position with respect to $L \setminus K_i$ so that it meets each other link component transversely.) Let $\hat \Sigma_i$ be the union of $\Sigma_i$ with the core disk $D_i = D^2\times 0$ inside the 2-handle attached along $K_i$. The homology classes $[\hat \Sigma_i]$ of these closed surfaces $\hat \Sigma_i$ form a basis for $H_2(X)$. (Indeed,  the homotopy equivalence collapsing  $X$ to $\vee^n S^2$ carries $\hat \Sigma_i$ to the $i^\text{th}$ wedge summand $S^2$ via a degree-one map.)

		For $i \neq j$, any intersection between $\hat \Sigma_i$ and $\hat \Sigma_j$ must occur in $S^3$. Let $\Sigma_j'$ be a surface obtained from $\Sigma_j$ by pushing the interior $\mathring{\Sigma}_j$ slightly into $B^4$ so that $\Sigma_j' \cap \partial B^4 = K_j$. This modified closed surface  $\hat \Sigma_j'= \Sigma_j' \cup_{K_j} D_j$  satisfies $[\hat \Sigma_j']=[\hat \Sigma_j]$ and is transverse to $\hat \Sigma_i$, meeting it only where $K_j$ intersects $\Sigma_i$. Thus we have
		\begin{align*}
			[\hat \Sigma_i]\cdot[\hat \Sigma_j] = [\hat \Sigma_i]\cdot[\hat \Sigma_j'] =\hat \Sigma_i \cdot \hat \Sigma_j' = \Sigma_i \cdot K_j = \lk(K_i,K_j).
		\end{align*}
		
		To find $[\hat \Sigma_i]\cdot [\hat \Sigma_i]$, we  need a perturbed copy of $\hat \Sigma_i$ that is transverse to $\hat \Sigma_i$. In the 2-handle $D^2 \times D^2$ attached along $K_i$, choose a perturbed copy of the core disk, expressed as  $D_i' = D^2 \times \{\epsilon\}$ for a point $\epsilon \in D^2 \setminus \{0\}$. The intersection $D_i' \cap S^3$ is an $n$-framed pushoff of $K_i$, which we will denote by $K_i'$. Let $\Sigma_i'$ denote a copy of $\Sigma_i$ dragged along during the isotopy from $K_i$ to $K_i'$, with its interior pushed into $B^4$ as above. Then $\hat \Sigma_i'=\Sigma_i' \cup D_i'$ satisfies $[\hat \Sigma_i']=[\hat \Sigma_i]$ and is transverse to $\hat \Sigma_i$, allowing us to compute
		$$\qquad \qquad
			[\hat \Sigma_i]\cdot[\hat \Sigma_i] = [\hat \Sigma_i]\cdot[\hat \Sigma_i'] =\hat \Sigma_i \cdot \hat \Sigma_j' = \Sigma_i \cdot K_i' = \lk(K_i,K_i')=\operatorname{framing}(K_i)\in \Z. \  \qed $$
\renewcommand{\qedsymbol}{}		
	\end{proof}
	\renewcommand{\qedsymbol}{$\square$}
	
	\vspace{-3.125ex}
	
	While it may seem restrictive, many interesting examples of 4-manifolds can be built by attaching 2-handles to $B^4$. If a 4-manifold $X$ has a unique 0-handle and some number of 4-handles, we may still apply Proposition~\ref{prop:linking}  because removing (the interiors of) the 4-handles does not affect $H^2(X)$; see Exercise~\ref{exer:glue} below. Similar techniques apply   in the presence of 1- and 3-handles. In this case, the 2-handles' attaching curves still provide a basis for the chain group $C_2(X)$; we must extract a basis for $H_2(X)$ as in \S\ref{subsec:homology}, with each basis element expressed as a linear combination of the 2-handles' attaching curves. Restricting the linking form to the subspace spanned by these elements yields a presentation of the intersection form on $H_2(X)$.

	\begin{exercise} \label{ex:represent} Let $X$ be any smooth, compact 4-manifold (possibly with boundary). Prove that each class in $H_2(X)$ is represented by a smoothly embedded surface in $X$.
		
		\emph{Hint: After choosing a handle structure on $X$, each class in $H_2(X)$ corresponds to some linear combination of 2-handles.   To build a representative surface, begin with an appropriate collection of disjoint copies of the cores of these 2-handles, then build the rest of your surface (while ensuring it is embedded).}
	\end{exercise}

	\begin{remark}\label{rem:transpose}
		We consider intersection forms up to congruence: If we change our basis for $\Z^n$ by a matrix $A$, then any bilinear form $Q$ on $\Z^n$ transforms as $Q \rightsquigarrow A^T Q A$. 
	\end{remark}

	\begin{exercise}\label{ex:sphere-bundle} Let $M_n$ denote the 4-manifold built from one 0-handle and two 2-handles as given in Figure~\ref{fig:hopf}. 
	
	\vspace{-0.5ex}
	
		\begin{enumerate}[itemsep=-0.5ex]
			\item[\textbf{(a)}] Show that $Q_{M_n} \cong \displaystyle \begin{bmatrix} 0 & 1 \\ 1 & n \end{bmatrix}$ (for a natural  choice of basis).

			\item[\textbf{(b)}] Show that $\partial M_n \cong S^3$.
			\item[\textbf{(c)}] In either order, do the following:
			\begin{enumerate}
				\item [(\emph{i})] \emph{Linear algebra:} Find a basis for $H_2(M_n;\Z)$ in which $Q_{M_n}$ is presented by 
				$$\begin{bmatrix} { } 0 & 1 { } \\  { } 1 & 0 { }\end{bmatrix} \quad \text{if }  n \text{ is even}  \qquad \text{ or } \qquad \begin{bmatrix} { } 0 & 1 { } \\ { } 1 & 1 { } \end{bmatrix}  \quad \text{if }  n \text{ is odd}. $$ 
				
				\vspace{-0.5ex}
				
				\item[(\emph{ii})] \emph{Kirby calculus:} Use handle slides to show $M_n \cong M_{m}$  $\Leftrightarrow$ $m \equiv n \pmod 2$.
			\end{enumerate}
		\end{enumerate}
	\end{exercise}

	\begin{figure}[h!]\center
		\includegraphics[width=.2\linewidth]{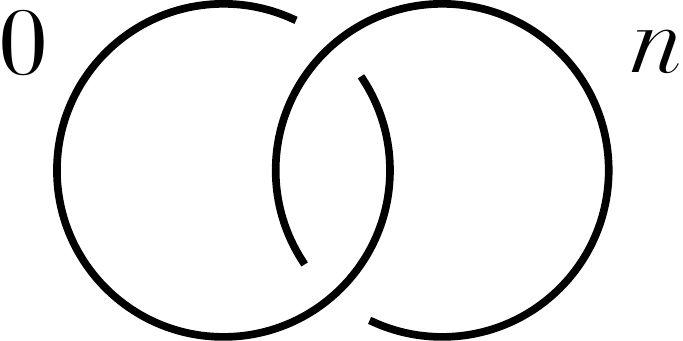}%
		\caption{A framed Hopf link.}\label{fig:hopf}
	\end{figure}

	The first part of the next exercise implies that we can also use Proposition~\ref{prop:linking} for 4-manifolds built from one 0-handle, some number of 2-handles, and a single 4-handle.
	
	\begin{exercise} \label{exer:glue}
		\begin{enumerate}
			\item[\textbf{(a)}] Given a closed 4-manifold $X$, show that $H_2(X \setminus \mathring{B}^4) \cong H_2(X)$ and $Q_{X \setminus B^4} \cong Q_X$. 
			\item[\textbf{(b)}] Given 4-manifolds $X$ and $X'$, show that $H_2(X \# X') \cong H_2(X) \oplus H_2(X')$ and $Q_{X \# X'} \cong Q_X \oplus Q_{X'}$. 
			\item[\textbf{(c)}] Generalize this to the case where $\partial X= - \partial X'$ is a 3-manifold $Y$ with $H_1(Y)=0$  and we glue $X \cup_Y -X'$.
		\end{enumerate}
	\end{exercise}

	\begin{example} \textbf{(a)} $\pm \CP^2$ (where $-\CP^2=\overline{\mathbb{C}P}^2$) is obtained from $B^4$ by attaching a $\pm1$-framed 2-handle along an unknot, then attaching a 4-handle, so $Q_{\CP^2} = \langle \pm 1\rangle$. 
		
		\textbf{(b)} $S^2 \times S^2$ corresponds to a 0-framed Hopf link (i.e., $n=0$ in Exercise~\ref{ex:sphere-bundle}) plus a 4-handle, so $Q_{S^2 \times S^2}$ is given by the \emph{standard hyperbolic form}
		$$H =\begin{bmatrix} 0 & 1 \\ 1 & 0 \end{bmatrix}.$$
	\end{example}

	\begin{exercise}
		Prove (or at least convince yourself) that every symmetric matrix $Q$ over $\Z$ arises as the linking matrix of a framed link in $S^3$, and hence that $Q$ arises as the intersection form of a compact 4-manifold with boundary.
	\end{exercise}

	\bigskip
	\medskip
	
	\addtocounter{section}{1}
	\section*{\textbf{Lecture 2: Intersection forms and smooth 4-manifolds}}
	\setcounter{subsection}{0}
	\setcounter{theorem}{0}

	\medskip
		
 I once had the pleasure of taking a topology class taught by Paul Melvin. I was an undergrad, and I confessed to him that I couldn't really picture 3-manifolds other than $\R^3$ --- and maybe $S^3$, too. He nodded consolingly and paused a while, then said, ``Well, I like to picture them as the boundaries of 4-manifolds.''
	
This lecture will take a closer look at the intersection forms of 4-manifolds, including 4-manifolds with boundary. And Paul was right: we will encounter an important interplay between 3- and 4-manifold topology, beginning with Rohlin's theorem \cite{rohlin} and the Poincar\'e homology sphere, and connecting all the way through to Manolescu's relatively recent disproof of the Triangulation Conjecture \cite{manolescu}.

\medskip
\smallskip
	
	{2.0. \textbf{The goal.}} 	
	Our goal is to outline the existence of closed, simply connected topological 4-manifolds that do not admit smooth structures. We will be able to sketch this after laying out some (elementary) properties of symmetric bilinear forms and citing some (non-elementary!) results about the ones that arise as intersection forms of smooth, closed, simply-connected 4-manifolds.

	\subsection{Basic properties  for intersection forms} We begin with  observations and definitions that apply to any symmetric bilinear form $Q$ on $\Z^n$.
	
	\begin{itemize}
		\item The \emph{rank} of $Q$ is simply $n$, the dimension of the underlying lattice. 
		\item Fix a basis for $\Z^n$ to view $Q$ as a symmetric matrix. The eigenvalues of a symmetric matrix are real, so we can let $b^+(Q)$ and $b^-(Q)$ denote the numbers of (strictly) positive and (strictly) negative eigenvalues, respectively, counted with multiplicity. Note that $b^\pm$ are independent of the basis for $\Z^n$, hence are invariants of $Q$ itself.
		
		\item  The \emph{signature} of $Q$ is $\sigma(Q)=b^+(Q)-b^-(Q)$.
		
		\item  $Q$ is $(\pm)$ \emph{definite} if $\sigma(Q) = \pm \rank(Q)$ and is \emph{indefinite} otherwise.
		
		\item  $Q$ is \emph{unimodular} if $\det Q = \pm 1$, which ensures $Q$ is invertible over $\Z$. More generally, we say $Q$ is \emph{nondegenerate} if $\det Q \neq 0$, which only ensures $Q$ is invertible over $\Q$.
		
		\item The \emph{type} of $Q$ is \emph{even} if $Q(x,x)\equiv 0 \! \mod 2$ for all $x \in \Z^n$ and is \emph{odd} otherwise.
	\end{itemize}
	
	\begin{remark}\label{rem:basis}
		The determinant of a bilinear form over $\Z$ is well-defined. This is because any change-of-basis matrix $A$ over $\Z$ is  unimodular, hence $\det(Q)=\det(A^T Q A)$.  When working with bilinear forms over $\Q$ or $\R$, this fails, but  the \textit{sign} of the determinant is still well-defined.  
	\end{remark}
	
	\begin{example}
		For the 4-manifold $X$ from Example~\ref{ex:linking}, we had
		$$Q_X =\left[\ \begin{matrix}  6 & 2 & 0 \\ 2 & 0 & 0 \\ 0 & 0 & -1  \end{matrix}\ \right].$$
		Observe that $\det(Q_X)=4$, so $Q_X$ is nondegenerate but not unimodular. We can easily see that $-1$ is an eigenvalue, so the positive determinant implies that we must have $b_-=2$ and $b_+=1$. Therefore $\sigma(Q_X)=-1$, and $Q_X$ is indefinite. Since $e_3 \cdot e_3=-1 \not \equiv 0$ mod 2, $Q_X$ is odd.
	\end{example}
	
	\begin{exercise}
		{\normalfont \textbf{(a)}} Repeat the above for the intersection forms on $\pm \mathbb{C}P^2$ and $S^2 \times S^2$. 
		
		\qquad {\normalfont \textbf{(b)}} How do these invariants behave under connected sums of 4-manifolds?
	\end{exercise}
	
	\begin{exercise}
		A form $Q$ on $\Z^n$ is even $\iff$ every matrix for $Q$ has even diagonal $\iff$ at least one matrix for $Q$ has even diagonal.
	\end{exercise}

	\begin{figure}[h]
	\center
		\includegraphics[width=.225\linewidth]{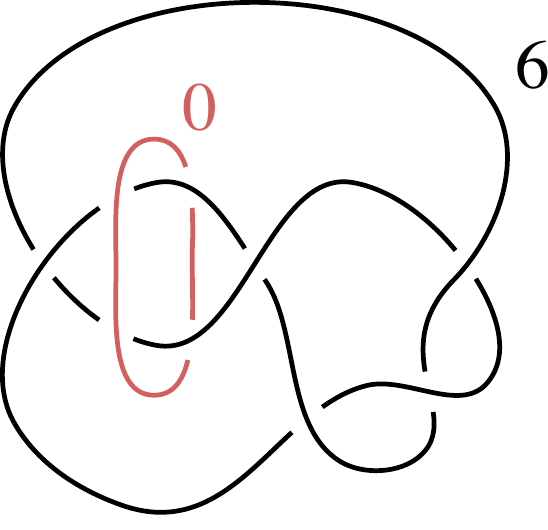}
		\caption{A 4-manifold with even intersection form.}\label{fig:blown}
	\end{figure}

	\begin{example}\label{ex:even}
		Take the 4-manifold $X$ in Example~\ref{ex:framed}, whose intersection form was not even, and blow down the $(-1)$-framed unknotted component. We obtain the handle diagram shown in Figure~\ref{fig:blown}, representing a new 4-manifold $X'$ with $\partial X'=\partial X$. Using Remark~\ref{rem:basis} and the fact that all framing coefficients in the diagram are even, we conclude that $X'$ is even.
	\end{example}

	We will not prove the following result, but we will briefly make use of it later. 
	
	\begin{theorem}\label{thm:bound}
		Every closed, connected, orientable 3-manifold bounds a smooth, simply connected 4-manifold  $X$ built from one 0-handle and a number of 2-handles. Moreover, we may choose $X$ so that $Q_X$ is even. 
	\end{theorem}
	
	The key idea behind realizing 3-manifolds as boundaries of 4-manifolds is to show that any closed, orientable 3-manifold $Y$ can be realized as Dehn surgery on an integrally-framed link $L$ in $S^3$. Then view $L$ as instructions for attaching 2-handles to $B^4$, yielding a 4-manifold $X$ with $\partial X=Y$. If $X$ has $Q_X$ odd, it turns out that there must be a particular type of non-empty sublink of $L$ (called a ``characteristic sublink'') that is responsible for $Q_X$ being odd, like the $(-1)$-framed unknot in Example~\ref{ex:even}. As in that example, there is always a sequence of blow-ups and blow-downs that eventually eliminate all nonempty characteristic sublinks, resulting in a new 4-manifold with even intersection form.

	The simplest symmetric bilinear form that is unimodular, positive definite, and even  turns out to be an 8-dimensional form known as $E_8$, presented by the matrix
	
	$$
	E_8 \  = \  \left[\ \begin{matrix}
		2 & 1 & 0 & 0 & 0 & 0 & 0 & 0
		\\
		1 & 2 & 1 & 0 & 0 & 0 & 0 & 0
		\\
		0 & 1 & 2 & 1 & 0 & 0 & 0 & 0 
		\\
		0 & 0 & 1 & 2 & 1 & 0 & 0 & 0 
		\\
		0 & 0 & 0  & 1 & 2 & 1 & 0 & 1
		\\
		0 & 0 & 0 & 0 & 1 & 2 & 1 & 0 
		\\
		0 & 0 & 0 & 0 & 0 & 1 & 2 & 0
		\\
		0 & 0 & 0 & 0 & 1 & 0 & 0 & 2
	\end{matrix}\ \right].
	$$
	
	\bigskip
	
	\begin{example}\label{ex:E8}
		We can realize $E_8$ as the intersection form of the compact 4-manifold with boundary shown in the top left corner of Figure~\ref{fig:E8}, known as the $E_8$-plumbing manifold. The boundary of the $E_8$-plumbing manifold is known as the \emph{Poincar\'e homology sphere}, denoted $P$. In Figure~\ref{fig:E8}, we see that $P$ can also be obtained from $(+1)$-framed Dehn surgery on the right-handed trefoil.
	\end{example}
	
	\begin{figure}\center
		\def\svgwidth{\linewidth} 
		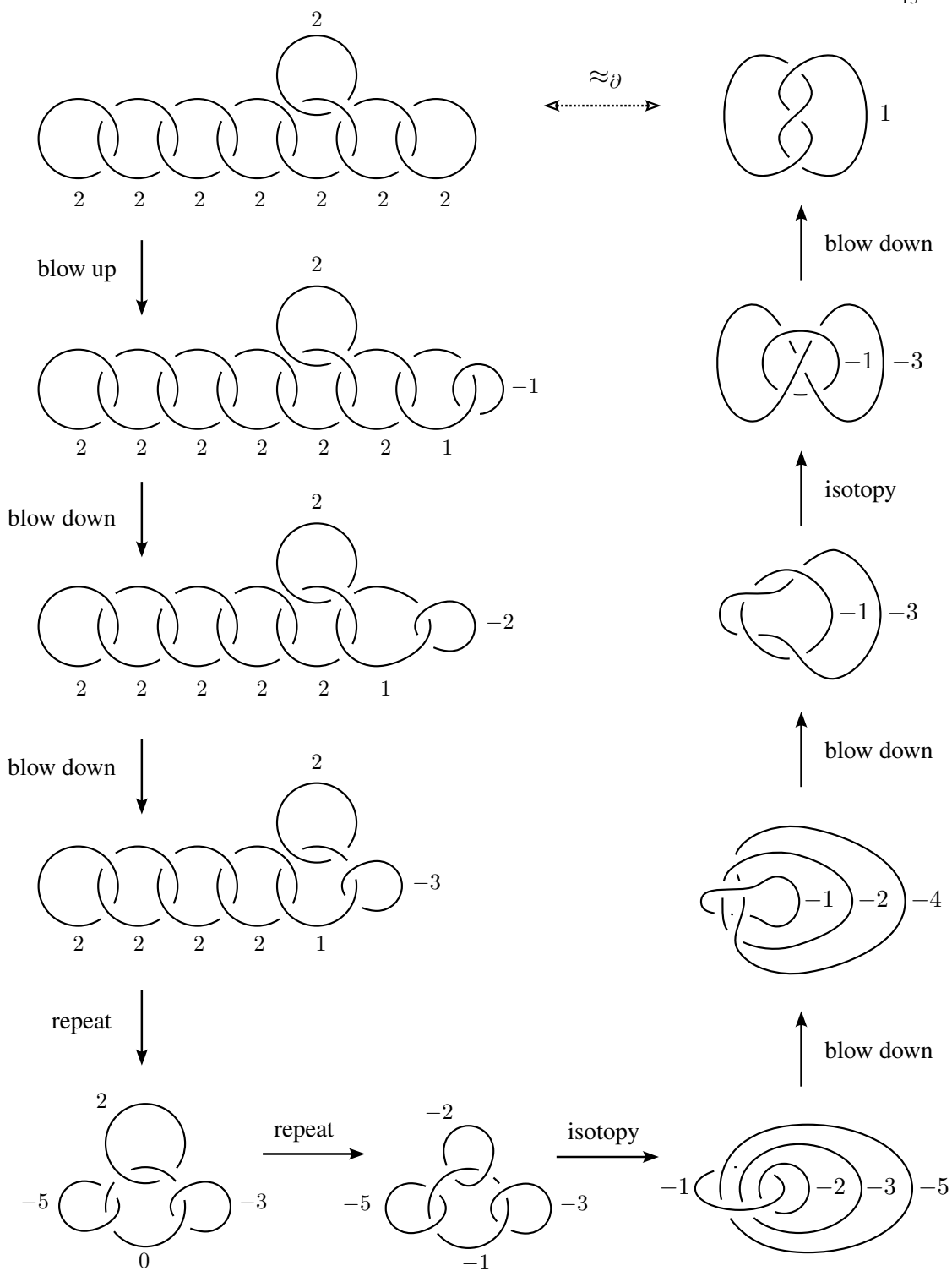 
		\medskip
		\caption{Relating the boundaries of the $E_8$-plumbing manifold and surgery on the trefoil knot.}\label{fig:E8}
	\end{figure}
	
	\begin{exercise}
		Carefully check the handle calculus in Figure~\ref{fig:E8}, paying  attention to the framings (especially near the end).
	\end{exercise}

	\begin{definition}
		Let $Y$ be a closed 3-manifold. We say that $Y$ is an \emph{integer homology 3-sphere} (abbreviated as $\Z HS^3$) if $H_*(Y;\Z) \cong H_*(S^3;\Z)$. More generally, given a coefficient ring $R$, we say $Y$ is an \emph{$R$-homology 3-sphere} if $H_*(Y;R) \cong H_*(S^3;R)$.
	\end{definition}

	\begin{exercise} {\normalfont\textbf{(a)}}
		$Q_X$ is unimodular $(\det(Q_X) = \pm 1)$ if and only if $\partial X$ is a $\Z HS^3$.
		
		\quad {\normalfont\textbf{(b)}}
		$Q_X$ is nondegenerate $(\det(Q_X)  \neq 0)$ if and only if  $\partial X$ is a $\Q HS^3$.
	\end{exercise}

	\subsection{Constraints on intersection forms} In what follows, $Q$ denotes a unimodular, symmetric bilinear form over $\Z$. First, a constraint that's purely algebraic:
	
	\begin{theorem}
		If $Q$ is even, then $\sigma(Q)$ is divisible by 8.
	\end{theorem}
	
	Note that the $E_8$ form satisfies this, of course, because $\sigma(E_8)=8$. Rohlin proved a stronger constraint for intersection forms of smooth 4-manifolds.
	
	\begin{theorem}[Rohlin \cite{rohlin}]\label{thm:rohlin} If $X$ is a smooth, closed, simply connected 4-manifold with even intersection form $Q_X$, then $\sigma(Q_X)$ is divisible by 16.
	\end{theorem}
	\begin{remark} Rohlin's theorem holds for any
            closed \emph{spin} 4-manifold. For a simply connected 4-manifold $X$ (possibly with nonempty boundary), the spin condition is equivalent to $Q_X$ being even.
	\end{remark}
	
	Note that Rohlin's theorem does not hold for 4-manifolds with nonempty boundary; indeed, the $E_8$-plumbing manifold from Example~\ref{ex:E8} provides a counterexample. Moreover, it follows that the Poincar\'e homology sphere cannot bound a smooth integer homology ball, otherwise we could glue it to the $E_8$-plumbing manifold to produce a smooth, closed 4-manifold that violates Rohlin's theorem (see Exercise~\ref{exer:glue}(c)). Remarkably, Freedman showed that this ``closing off'' strategy can always be applied in the topological setting.
	
	\begin{theorem}[Freedman \cite{freedman}]\label{thm:contractible}
		Every integer homology sphere bounds a compact, contractible topological 4-manifold.
	\end{theorem}
	
	Given an integer homology 3-sphere $Y$, the proof of this theorem in \cite{freedman-quinn} goes roughly as follows: Beginning with $Y \times [-1,1]$, one can perform a sequence of surgeries on loops and topologically embedded spheres in the interior of $Y \times [-1,1]$ (i.e., replacing copies of $S^1 \times D^3$ with $D^2 \times S^2$, or vice versa) to produce a new 4-manifold $W$ with  $\partial W = Y \sqcup -Y$, $\pi_1(W)=0$, and $H_2(W)=0$. Next, glue infinitely many of these end-to-end to produce an infinite stack  $W \cup W \cup \cdots$ and take its 1-point compactification. The resulting  space has boundary $Y$, is contractible, and can be shown to be a topological manifold (i.e., even at the compactification point).

	\begin{proof}[Sketch for Theorem~\ref{thm:non-smoothable}]
		Let $X$ denote the $E_8$-plumbing manifold. Since $E_8$ is unimodular, $\partial X$ is a $\Z HS^3$ and thus bounds some   contractible topological 4-manifold  $C$  (Theorem~\ref{thm:contractible}). Then  $X \cup -C$ is a closed, simply connected topological 4-manifold that still has intersection form $E_8$ (because $C$ is contractible). By Rohlin's theorem (Theorem~\ref{thm:rohlin}), this 4-manifold $X \cup -C$ cannot admit a smooth structure.
	\end{proof}

	\begin{exercise}
		Find another closed, simply connected topological 4-manifold (other than $E_8$) that does not admit a smooth structure.
	\end{exercise}
	
	\subsection{Addendum: Rohlin's invariant and beyond}
	
	As mentioned in Theorem~\ref{thm:bound}, every closed, orientable, connected 3-manifold $Y$ bounds a smooth, simply connected, compact 4-manifold with $Q_X$ even. Provisionally, we can then define a quantity
	$$\mu(Y)=\frac{1}{8} \sigma(X) \mod 2.$$
	When $Y$ is a $\Z HS^3$, this quantity is an invariant of $Y$: If $Y$ bounds $X$ and $X'$ with the desired properties, then applying Rohlin's theorem to $X \cup -X'$ gives$$\sigma(X)-\sigma(X')=\sigma(X \cup -X') \equiv 0 \mod 16 \implies \frac{1}{8}\sigma(X) \equiv \frac{1}{8}\sigma(X') \mod 2.$$
	
	\begin{remark}We can obtain an invariant in the general setting  where  $H_1(Y) \neq 0$ by considering \emph{spin} structures. The even type of $X$ ensures that $X$ admits a spin structure, and it is unique if $H^1(X;\Z/2\Z)=0$. The spin structure on $X$ restricts to a spin structure $s$ on $Y$. The quantity $\mu(Y,s)$ is defined as above, but only  using 4-manifolds $X$ over which $s$ extends.
	\end{remark}
	
	\begin{example}
		For the Poincar\'e homology sphere $P=\Sigma(2,3,5)$ we can use the $E_8$-plumbing as $X$, so
		$$\mu(P)=\frac{1}{8} \sigma(E_8) =\frac{1}{8} (8) = 1 \pmod 2$$
	\end{example}

	\begin{exercise}
		If $Y$ is a $\Z HS^3$ with $\mu(Y)\equiv 1$, then $Y$ cannot bound a smooth integral homology 4-ball ($\Z HB^4$).
	\end{exercise}

	A priori, the obstruction from Rohlin's invariant $\mu$ does not obstruct $Y \# Y$ from bounding a $\Z HB^4$. However, in 2013, Manolescu used Pin(2)-equivariant Seiberg-Witten theory to prove:

	\begin{theorem}[Manolescu \cite{manolescu}]
		If $Y$ is a $\Z HS^3$ with $\mu (Y) \equiv 1 \! \mod 2$, then $Y \# Y$ does not bound a $\Z HB^4$.
	\end{theorem}

	Combining this with work of Galewski-Stern and Matumoto, this enabled Manolescu to resolve the higher-dimensional Triangulation Conjecture:

	\begin{theorem}[\cite{manolescu}]
		In all dimensions $ \geq 5$, there exists a closed, connected, topological manifold that is not triangulable (i.e., not homeomorphic to any simplicial complex).
	\end{theorem}

	\medskip
	\bigskip
	
	\section*{\textbf{Lecture 3: The minimal genus problem}}

	\medskip	
	\addtocounter{section}{1}
	\setcounter{subsection}{0}
	\setcounter{theorem}{0}

	Previous lectures covered the core algebraic topology of smooth 4-manifolds and surveyed some important theorems that (to a significant degree) reduce the study of closed, simply connected, topological 4-manifolds to algebraic topology. In this lecture, we will see how embedded surfaces can shed light on the topology of a 4-manifold $X$ --- seeing beyond the algebraic topology and continuous topology, peering directly into the smooth structure on $X$.  The central question in this lecture is motivated by the fact that any second homology class  in a smooth 4-manifold can be represented by a smoothly embedded surface (cf. Exercise~\ref{ex:represent}).
	
	\begin{problem}[The Minimal Genus Problem]
		Given a smooth 4-manifold $X$ and a class $\alpha \in H_2(X;\Z)$, what is the minimal genus realized by a closed, orientable, smoothly embedded surface $\Sigma \subset X$ with $[\Sigma]= \alpha \in H_2(X;\Z)$?
	\end{problem}
	
\smallskip
	
	{3.0. \textbf{The goal.}} 	
		We will practice describing embedded surfaces (and their tubular neighborhoods) in 4-manifolds. We will review various constraints on their genera, culminating in a sketch of the following landmark theorem. 
	
	\begin{theorem}[Donaldson \cite{donaldson}, Freedman \cite{freedman}]\label{thm:exoticR4}
		There exist smooth 4-manifolds that are homeomorphic but not diffeomorphic to $\R^4$.
	\end{theorem}
	
	In contrast, in all other dimensions $n \neq 4$, there is a unique smooth structure on $\R^n$. 
	
	\subsection{The minimal genus problem} We begin with a few examples.
	
	\begin{example}
		For $X=S^2 \times S^2$,  the natural generators $\alpha=[S^2 \times \{pt\}]$ and $\beta=[\{pt\} \times S^2]$ of $H_2(X) \cong \Z \oplus \Z$ are  represented by smoothly embedded 2-spheres.
	\end{example}
	
	\begin{exercise} Consider $S^2 \times S^2$ and the generators  $\alpha,\beta \in H_2(S^2 \times S^2)$ as above.
		
		\begin{enumerate}
			\item[\textbf{(a)}] What is the homology class of the diagonal embedding $S^2 \hookrightarrow S^2 \times S^2$ given by $x \mapsto (x,x)$? What about $x \mapsto (x,-x)$ and $x\mapsto (-x,x)$?
			\item[\textbf{(b)}] For $\alpha,\beta \in H_2(S^2 \times S^2)$ as above, try to find or estimate the minimal genus of $\alpha + n \beta$ for $n \in \Z$. 
		\end{enumerate}
	\end{exercise}

Let us consider another important 4-manifold: $\CP^2$. The generator of $H_2(\CP^2) \cong \Z$ is represented by the complex projective line $\CP^1 \subset \CP^2$, which is a smooth 2-sphere.  What about the class $d [\CP^1] \in H_2(\CP^2)$ for $d \geq 2$? The next result asserts that a minimal representative is always realized by a degree-$d$ complex curve $\Sigma_d$ given by 
	$$\{[x : y : z] \in \CP^2 : x^d + y^d + z^d = 0 \}.$$
	
	\smallskip
		
	\begin{theorem}[Thom Conjecture, Kronheimer-Mrowka \cite{k-m:thom}]\label{thm:thom}
		If $\Sigma$ is a smooth oriented surface embedded in $\CP^2$ representing $d[\CP^1] \in H_2(\CP^2)$ with $d\neq 0$, then $$g(\Sigma) \geq (d-1)(d-2)/2.$$
	\end{theorem}
	
	For example, if $d=2$, we get $g=0$, so $2[\CP^1]$ is also represented by a smooth 2-sphere. This is illustrated in Figure~\ref{fig:2CP1} using Kirby diagrams: Beginning with the standard diagram of $\CP^2$, take two disjoint copies of the core disk of the 2-handle (equipped with the same orientation). The boundaries of these disks form an oriented Hopf link $L$ in the boundary of the 0-handle, and we observe that $L$ bounds an annulus. Gluing this annulus to the two disks yields a smooth 2-sphere in the class $2[\CP^1]$.
	
	\begin{figure}
		\center
		\includegraphics[width=.85\linewidth]{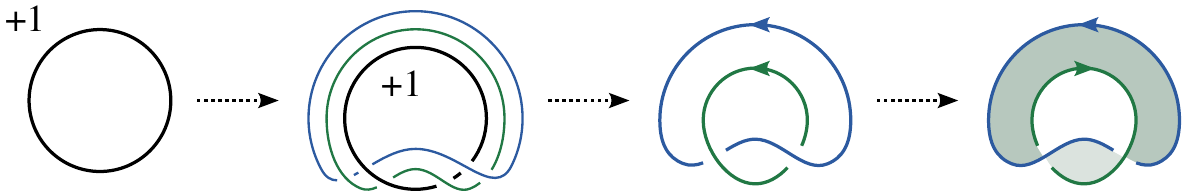}
		\caption{Constructing a smooth 2-sphere in $\CP^2$.}\label{fig:2CP1}
	\end{figure}
	
	While Theorem~\ref{thm:thom} was proven using gauge theory, the special case of $d=3$ (i.e., $g=1$) can be obtained using Rohlin's theorem (Theorem \ref{thm:rohlin}):
	
	\begin{proposition}\label{prop:cubic}
		The class $3[\CP^1] \in H_2(\CP^2)$ cannot be represented by a smoothly embedded 2-sphere.
	\end{proposition}
	
	\begin{proof}
		Towards a contradiction, suppose $\Sigma \subset \CP^2$ is a smooth 2-sphere with $[\Sigma]=3[\CP^1]$. Note that $$[\Sigma] \cdot [\Sigma] = (3[\CP^1]) \cdot (3[\CP^1]) = 9 ([\CP^1] \cdot [\CP^1])= 9,$$
		so  $\Sigma$ has  a tubular neighborhood diffeomorphic to the $D^2$-bundle over $S^2$ with Euler number $9$. This has the Kirby diagram shown on the left side of the figure below. Blowing up $\CP^2$ eight times yields $\CP^2 \# 8 \overline{\CP}^2$. In this new 4-manifold, we can consider the ``proper transform'' of $\Sigma$, a smooth 2-sphere $\Sigma'$ of square $+1$ in the homology class $3[\CP^1]+\sum_{i=1}^8 e_i$, where $e_i$ is the homology class of $\overline{\CP}^1$ in the $i^\text{th}$ copy of $\overline{\CP}^2$; this sphere $\Sigma'$ is depicted on the right side of Figure~\ref{fig:proper-transform}  as the union of a gray disk and the core of the $(+1)$-framed 2-handle. 
		
		\begin{figure}[h]
			\center
			\includegraphics[width=.45\linewidth]{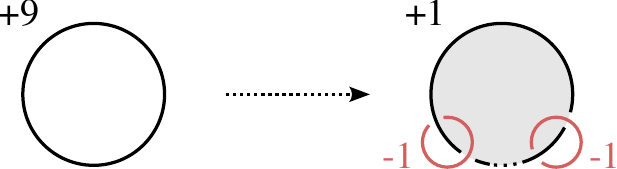}
			\caption{The proper transform of a sphere after blowing up.}\label{fig:proper-transform}
		\end{figure}
		
		Let $X$ denote the smooth, closed, simply connected 4-manifold obtained by blowing down $\Sigma'$ in $\CP^2 \# \, 8 \, \overline{\CP}^2$. It can be shown that the intersection form of $Q_X$ is the $E_8$ lattice (Exercise~\ref{ex:deg3}), hence $X$ has $\sigma=-8$ with an even intersection form, contradicting Rohlin's theorem.
	\end{proof}

	\begin{exercise}\label{ex:deg3}
		Let us calculate the intersection form of $X$. Since $\Sigma'$ is a sphere of square $+1$, it has a tubular neighborhood bounded by $S^3$, hence $H_2(\CP^2 \# 8 \overline{\CP}^2)$ splits into a subspace $V$ spanned by $v_0:=[\Sigma']$ and an orthogonal subspace $V^\perp$. After blowing down, we have $H_2(X) \cong \Z^8$ and $Q_X$ is the  restriction of $Q_{\CP^2 \# \, 8 \, \overline{\CP}^2}$ to $V^\perp$.

		\begin{enumerate}
			\item[\textbf{(a)}] Show that the classes $e_i$ do not lie in $V^\perp$, but the classes $v_i:= e_i-e_{i+1}$ do indeed lie in $V^\perp$ for $i=1,\ldots,7$. Also show that $v_i \cdot v_i = -2$.
			
			\item[\textbf{(b)}] Show that the class  $[\CP^1]$ does not lie in $V^\perp$, then find a linear combination   $v_8:=[\CP^1] + \sum_{i=1}^8 c_i e_i $ for some $c_i \in \Z$ that lies in $V^\perp$. Can you find a solution such that $v_8 \cdot v_8 = -2$? 
			
			\item[\textbf{(c)}] Verify that, with your choice of $v_8$ from (b), the collection $\{v_0, v_1,\ldots,v_8\}$ gives a basis for $H_2(\CP^2 \# 8 \overline{\CP}^2)$ (e.g., by showing that the associated change-of-basis matrix $B$ has determinant $\pm 1$).
			
			\item[\textbf{(d)}] In this new basis (up to reordering), show that $Q_{\CP^2 \# 8 \overline{\CP}^2}$ is given by $\begin{bmatrix} 1 & 0 \\ 0 & E_8 \end{bmatrix}$.  (Hint: Don't forget about Remark~\ref{rem:transpose}.)
			
			\item[\textbf{(e)}] Conclude that $Q_X \cong E_8$.
		\end{enumerate}
		
	\end{exercise}
	
	This is an example of a more general approach to minimal genus problems: If we wish to show that a second homology class $\alpha$ in a 4-manifold $X$  cannot be represented by a smooth surface of genus $g$, suppose that such a surface exists and use it to construct an auxiliary smooth 4-manifold that violates Rohlin's theorem or another smooth obstruction. An example of this concerns surfaces $\Sigma \subset X$ whose homology classes are \emph{characteristic}, i.e., $[\Sigma] \cdot [\Sigma'] \equiv [\Sigma']\cdot [\Sigma']$ mod 2 for all surfaces $\Sigma' \subset X$.
	
	\begin{theorem}[Kervaire-Milnor \cite{kervaire-milnor}]
		Let $X$ be a smooth, closed, simply connected 4-manifold. If $\Sigma \subset X$ is a smoothly embedded 2-sphere that is characteristic, then $$[\Sigma]\cdot[\Sigma]\equiv 0 \mod 16.$$
	\end{theorem}
	
	For completeness, we note that a vast generalization of the Thom Conjecture is true. Indeed, the moral of the Thom Conjecture is that nonsingular complex curves in $\CP^2$ are genus-minimizing in their homology classes (and in this case the genus happens to have a simple expression in terms of the degree of the homology class). More generally, we have the following:
	
	\begin{theorem}[Symplectic Thom Conjecture, Ozsv\'ath-Szab\'o \cite{os:symplectic}] 
		A closed, embedded symplectic surface in a closed, symplectic four-manifold is genus-minimizing in its homology class.
	\end{theorem}
	
	\subsection{The local minimal genus problem}
	
	Every knot $K$ in $S^3$ bounds embedded surfaces inside $S^3$, but we can often find simpler surfaces bounded by $K$ if we allow the interior of the surface to be embedded in $B^4$. This motivates the following definition:
	
	\begin{definition}
		Given a knot $K$ in $S^3$, a \emph{slice surface} for $K$ is a compact, connected, orientable surface $(\Sigma,\partial \Sigma) \hookrightarrow (B^4,S^3)$ with $\partial \Sigma = K$. The \emph{slice genus} of $K$ is
		$$
		g_4(K) = \min \left\{ g(\Sigma) \mid  \text{$\Sigma$ is a slice surface for $K$}\right\}.
		$$
	\end{definition}
	
	We say a knot $K$ is \emph{slice} if it has slice genus zero. Historically, this is motivated by the observation that slice knots are precisely those knots that arise as the transverse intersection (or slice) of  an embedded 2-sphere in $S^4$ along the equatorial $S^3 \subset S^4$. 
	
	\begin{example}
		A slice disk  bounded by the stevedore knot is depicted in Figure~\ref{fig:stevedore}. Here the disk is represented diagrammatically as a self-transverse immersed disk in $S^3$; the double-points occur as ``ribbon'' intersections, and we may obtain an embedded surface in $B^4$ by pushing the interior of the disk into $B^4$ such that the intersecting sheets occur at different heights.
	\end{example}
	
	\begin{figure}
		\center
		\includegraphics[width=.225\linewidth]{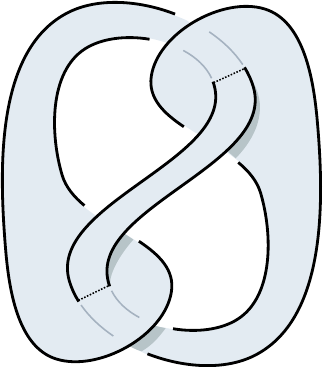}
		\caption{A ribbon-immersed disk bounded by the stevedore knot.}
		\label{fig:stevedore}
	\end{figure}
	
	\begin{example}
		The trefoil knot is \emph{not} slice, i.e., does not have slice genus zero. (Its slice genus is one.) To see this, consider a cuspidal degree-3 curve $\Sigma \subset \CP^2$; the trefoil arises as the link of the unique singularity of $\Sigma$. If the trefoil were smoothly slice, one could delete the conic singularity from $\Sigma$ and glue in the smooth slice disk to produce a smooth 2-sphere in the same homology class as $\Sigma$, namely $3[\CP^1] \in H_2(\CP^2)$. However, this would contradict Proposition~\ref{prop:cubic}.
	\end{example}

	\begin{remark}
		As discussed in \cite{acmps},  a related strategy can be used to show that the figure-eight knot is not smoothly slice; if it were, one could construct a smooth 2-sphere representing the class $(1,3)$ in $H_2(\CP^2 \# \CP^2;\Z) \cong \Z \oplus \Z$, the existence of which can be ruled out using Rohlin's theorem. Remarkably, Aceto-Castro-Miller-Park-Stipsicz \cite{acmps} show that a refinement of this approach can be used to show that the $(2,1)$-cable of the figure-eight knot is not smoothly slice, providing an alternate solution to a longstanding question (originally resolved in \cite{dkmps}). In place of Rohlin's theorem, this refinement uses obstructions from Seiberg-Witten theory on smooth spin 4-manifolds equipped with a $\Z/2\Z$-symmetry.
	\end{remark}

	We can also consider embedded surfaces in the continuous setting. To avoid pathological behavior, it is common to work with a somewhat restricted class of  topologically embedded, locally flat surfaces: 
	
	\begin{definition}[Topological embeddings]
		Let $\Sigma$ be a surface and $X$ a topological 4-manifold, each possibly with boundary. A map $\iota: \Sigma \hookrightarrow X$ is a \emph{topological embedding} if $\iota$ is a homeomorphism onto its image. The embedding is \emph{locally flat} if every point $p \in \iota(\Sigma) \subset X$ has a neighborhood in which the pair $(X,\iota(\Sigma))$ is homeomorphic to $(\R^4,\R^2 \times 0)$. 
		For brevity, we say a knot $K \subset S^3$ is \emph{topologically slice} if it bounds a topologically embedded, locally flat disk in $B^4$. 
	\end{definition}

	In many cases, there are more direct obstructions to sliceness (in both the smooth and topological settings). A key obstruction dates back to work of Fox and Milnor, who studied the question of which knots and links arise as links of singularities of piecewise-linear 2-spheres in $S^4$ \cite{fox-milnor}. If knots $K_1,\ldots,K_n$ form the links of singularities of such a sphere $\Sigma \subset S^4$, then the knot $K_1 \# \cdots \# K_n$ bounds a nonsingular disk in $B^4$. (To see this,  take a path $\gamma$ that joins the singularities,  then remove a neighborhood $U \cong \mathring{B}^4$ of $\gamma$ from $\Sigma$ to obtain a nonsingular disk $D \subset \Sigma$ that lies in $S^4 \setminus U \cong B^4$ and is bounded by $K_1 \# \cdots \# K_n$.) Fox and Milnor showed that the Alexander polynomial provides a slicing obstruction:
	
	\begin{theorem}[Fox-Milnor \cite{fox-milnor}]
		If $K$ is (smoothly or topologically) slice, then its Alexander polynomial factors as $\Delta_K(t)=f(t)f(1/t)$ for some polynomial $f \in \Z[t]$.
	\end{theorem}
	
	\begin{corollary}
		If $K$ is (smoothly or topologically) slice, then the value $\left|\Delta_K(-1)\right|$, known as the determinant of $K$, is a square.
	\end{corollary}
	
	\begin{example}
		The trefoil knot has Alexander polynomial $\Delta_K(t)=t-1+t^{-1}$, hence has $|\Delta_K(-1)|=3$ and is not smoothly or topologically slice.
	\end{example}
	
	\begin{exercise}\label{exer:942}
		Find the Alexander polynomial of the knot $9_{42}$ shown in the middle of Figure~\ref{fig:not942} and use it to prove that $9_{42}$ is not (smoothly or topologically) slice.
	\end{exercise}
	
	Remarkably, Freedman showed that the Alexander polynomial can provide a sufficient condition for topological slicing:
	
	\begin{theorem}[Freedman \cite{freedman}]
		If $\Delta_K(t) \equiv 1$, then $K$ is topologically slice.
	\end{theorem}
	
	\begin{example}\label{ex:pretzel}
		The pretzel knot $P(3,-5,-7)$ depicted in Figure~\ref{fig:pretzel} has trivial Alexander polynomial $\Delta(t) \equiv 1$, hence is topologically slice.
	\end{example}
	
	\begin{figure}
		\center
		\includegraphics[width=.5\linewidth]{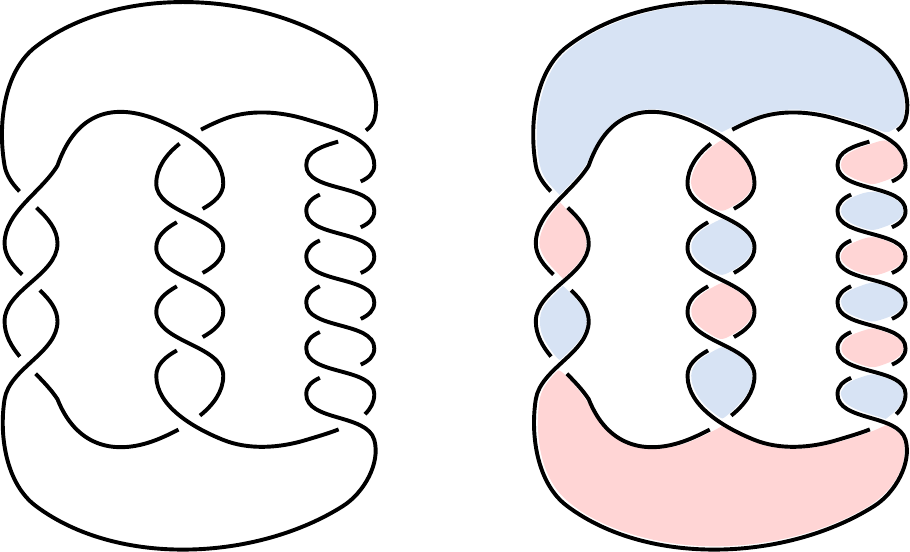}
		\caption{The pretzel knot $P(3,-5,-7)$ and a genus-one Seifert surface.}
		\label{fig:pretzel}\end{figure}
	
	This contrasts with the smooth setting:
	
	\begin{proposition}[\cite{rudolph:qp-obstruction}]\label{prop:pretzel}
		The pretzel knot $P(3,-5,-7)$ is not smoothly slice.
	\end{proposition}

	\begin{proof}
		The strategy will build on the approach from previous examples: We will first embed a genus-one Seifert surface bounded by $P(3,-5,-7)$ from Figure~\ref{fig:pretzel} into a smooth, closed surface $\Sigma$ of degree $d=6$ in $\CP^2$ that minimizes the genus in its homology class. Then we will show that, if  $P(3,-5,-7)$ were smoothly slice, we could modify $\Sigma$ to produce a smooth surface of smaller genus in the same homology class, violating  Theorem~\ref{thm:thom}. 
		
		To that end, we observe that the pretzel knot $P(3,-5,-7)$ can be expressed as the closure of a braid (in particular, a \emph{strongly quasipositive braid} \cite{rudolph:braided-surface,rudolph:kauffman-bound}) as illustrated in Figure~\ref{fig:sqp-pretzel}. The right side of Figure~\ref{fig:sqp-pretzel} depicts a natural genus-one Seifert surface $F$ associated to this braid.

	\begin{figure}[b]
		\center
		\includegraphics[width=.975\linewidth]{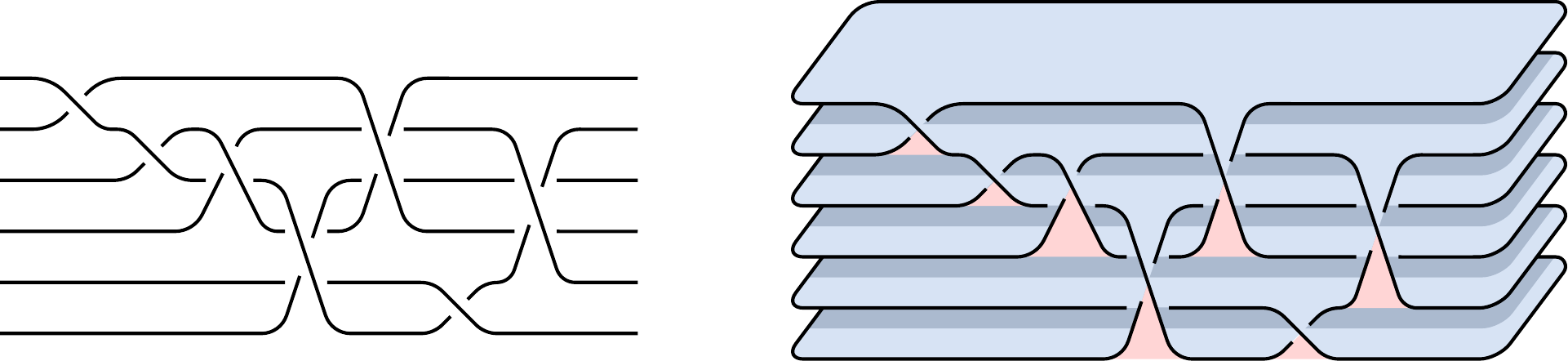}
		\caption{Expressing the pretzel knot $P(3,-5,-7)$ as a braid, together with a braided Seifert surface.}
		\label{fig:sqp-pretzel}\end{figure}
		
		By attaching bands to $F$ and, where necessary, sliding the original  bands over the new bands as in Figure~\ref{fig:band-attachment}, we can embed $F$ into the standard Seifert surface $F'$ for the torus link $T(6,6)$ depicted in Figure~\ref{fig:torus-braid}. Note that $g(F')=10$. This torus link arises as the link of the unique singularity for a cuspidal degree-6 curve in $\CP^2$. We may replace a neighborhood of the singularity with $F'$ (and smooth corners) to produce a smooth surface $\Sigma$ of genus $g(\Sigma)=g(F')=10$
		representing $6[\CP^1] \in H_2(\CP^2)$.
		
		If $P(3,-5,-7)$ were smoothly slice, then we could replace the genus-one subsurface $F \subset \Sigma$ with a slice disk  and produce a smooth, closed surface of genus 9 representing $6[\CP^1] \in H_2(\CP^2)$,  violating Theorem~\ref{thm:thom}.
	\end{proof}

	\begin{figure}
		\center
		\includegraphics[width=.65\linewidth]{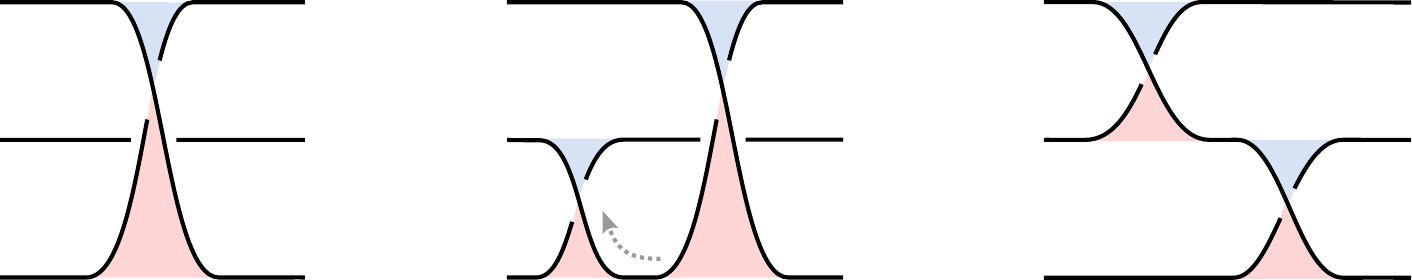}
		\caption{Adding a band and performing a band slide.}
		\label{fig:band-attachment}\end{figure}

	\begin{figure}
		\center
		\includegraphics[width=.55\linewidth]{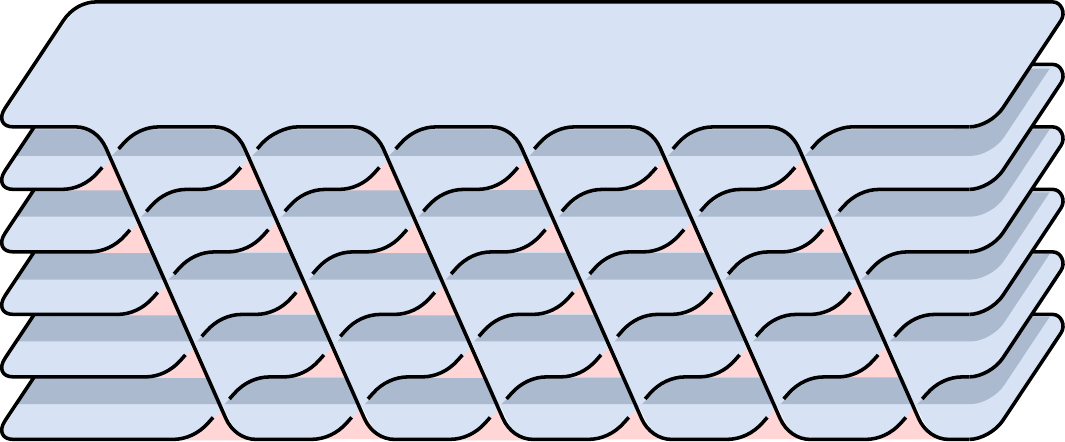}
		\caption{The standard Seifert surface for the torus link $T(6,6)$.}
		\label{fig:torus-braid}\end{figure}
	
	\begin{remark}
		Combined with Example~\ref{ex:pretzel}, the preceding argument shows that the Thom Conjecture fails in the topological category. In particular, since $P(3,-5,-7)$ is topologically slice, the construction above can be used to produce a topologically embedded, locally flat surface of genus 9 representing $6[\CP^1]\in H_2(\CP^2)$.
	\end{remark}

	\begin{remark}
		For completeness, here is the full statement from \cite[Theorem 1.13]{freedman}: \emph{Consider a smooth embedding $f:S^1 \hookrightarrow S^3$. The knot $K=f(S^1)$ has Alexander polynomial $\Delta_K(t) =1$ if and only if $f$ extends to some topological embedding $F: D^2 \hookrightarrow B^4$ such that}
		\begin{enumerate}
			\item \emph{$F$ is smooth except at a single point $p \in D^2$ which is nevertheless local-homotopically unknotted (or, equivalently, the end of $B^4 \setminus F(D^2)$ is proper homotopy equivalent to the standard end $S^1 \times D^2 \times [0,\infty)$);}
			
			\item \emph{$\pi_1\big(B^4 \setminus F(D^2)\big) \cong \Z$.}
		\end{enumerate}
	\end{remark}

	\subsection{An exotic $\boldsymbol{\mathbf{R}^4}$} To sketch the existence of an exotic $\R^4$, we will need a  result that ensures the 4-manifold we construct admits a smooth structure at all. 
	
	\begin{theorem}[Quinn \cite{quinn:ends}]\label{thm:quinn}
		Any connected  noncompact topological 4-manifold admits a smooth structure (which extends the unique smooth structure up to diffeomorphism on any boundary components).
	\end{theorem}
	
	This result in hand, we can sketch the existence of an exotic $\R^4$.

	\begin{proof}[Sketch for Theorem~\ref{thm:exoticR4}.]
		Let $K$ be any knot in $S^3$ that is topologically but not smoothly slice in $B^4$; for example, one may take $K=P(3,-5,-7)$ by Example~\ref{ex:pretzel} and Proposition~\ref{prop:pretzel}. Viewing $S^4$ as a union of two 4-balls $B^4_-$ and $B^4_+$, we write $$\R^4 = S^4\setminus\infty = B^4_-\cup (B^4_+\setminus 0)$$
		
		Let $K$ be embedded in the equatorial $S^3=\partial B^4_\pm$, and let $\Delta \subset (B^4_+ \setminus 0)$ be a topological slice disk bounded by $K$. Since $\Delta$ is a topologically embedded locally flat disk, it has a tubular neighborhood $N(\Delta)\cong \Delta \times D^2$. We may view $N(\Delta)$ as a 2-handle and attach it to $B^4_-$ to produce a topological embedding of a 4-manifold $X=B^4 \cup_{N(K)} D^2 \! \times \! D^2$ into $\R^4$. (This is an example of a \emph{knot trace}; see Exercise~\ref{exer:trace}.)  Let $f: X \hookrightarrow \R^4$ denote this embedding and set $E= \R^4 \setminus f( \mathring{X})$. 
		
		Now let us define an alternative smooth structure on the underlying topological manifold $\R^4$, yielding a new smooth 4-manifold $\mathcal{R}$: On $f(X) \subset \R^4$, the smooth structure is defined to be the pushforward of the standard smooth structure on $X=B^4 \cup_{N(K)} D^2 \! \times \! D^2$. By Theorem~\ref{thm:quinn}, we can extend this to a smooth structure on $E=\R^4 \setminus f(\mathring{X})$. By construction, this new smooth 4-manifold $\mathcal{R}$ is homeomorphic to $\R^4$ (since we have not altered the underlying topological manifold). 
		
		Now observe that $\mathcal{R}$ contains an embedded 2-sphere formed from the cone on $K$ inside $B^4_-$ and a copy of the core disk $D^2 \! \times \! 0$ inside the 2-handle in $D^2 \! \times \! D^2$. This 2-sphere is smooth away from its unique conical singularity over $K$. If $\mathcal{R}$ were diffeomorphic to $\R^4$, then the Fox-Milnor observation discussed above would imply that $K$ is smoothly slice in $B^4$, a contradiction. We conclude that $\mathcal{R}$ is an exotic  $\R^4$.
	\end{proof}

	As a historical remark, we note that the existence of an exotic $\R^4$ was originally proven using a different construction and obstruction. In particular, an exotic $\R^4$ is found as a certain subspace of the \emph{Kummer surface}, commonly known as $K3$, which is given by the equation
	$$K3 = \{[z_0:z_1:z_2:z_3] \in \mathbb{C}P^3 \, \mid \, z_0^4 +z_1^4+z_2^4+z_3^4 =0\}.$$
	
	The relevant obstruction comes from Donaldon's Diagonalization Theorem:
	
	\begin{theorem}[Donaldson \cite{donaldson}]
		Let $X$ be a smooth, closed, oriented 4-manifold. If the intersection form $Q_X$ is definite, then $Q_X$ is diagonalizable (hence equivalent to  $\pm I_n$).
	\end{theorem}

	\subsection{Addendum and additional exercises} We conclude this lecture with some additional material and exercises that focus on the relationship between smooth and singular surfaces.
	
	\begin{exercise}
		[Neighborhoods of PL spheres] Suppose that $\Sigma$ is a 2-sphere in a smooth 4-manifold $X$ that is smoothly embedded away from a finite number of singularities, each modeled on the cone of a knot in $B^4$. Show that $\Sigma$ has a neighborhood $N(\Sigma)$ that is diffeomorphic to the $n$-framed trace of some knot $K \subset S^3$, where $n$ is the self-intersection number $[\Sigma]\cdot [\Sigma]$. 
	\end{exercise}

	The next exercise presents a useful lemma that is closely related to Fox and Milnor's observation about piecewise-linear 2-spheres in $\R^4$  \cite{fox-milnor}.
	
	\begin{exercise}[Trace embedding lemma]
		Let $X(K)$ denote the 0-framed trace of a knot $K \subset S^3$. Show that $X(K)$ embeds smoothly into $\R^4$ if and only if $K$ is slice.
	\end{exercise}
	
	\begin{remark}
		In \cite{picc:conway}, Piccirillo used the Trace Embedding Lemma to help prove that the Conway knot (denoted  $11n{34}$) is not slice, resolving a longstanding question. All prior slicing obstructions fail to obstruct the Conway knot from being slice directly, including the $s$-invariant (which vanishes for all slice knots) defined by Rasmussen \cite{rasmussen:s} via Khovanov homology \cite{khovanov00}. Piccirillo's elegant proof circumvents this by identifying a second knot $K$ whose 0-trace is diffeomorphic to that of the Conway knot, and  then she shows that $s(K) \neq 0$; this implies $K$ is not slice, hence $X_0(K)$ cannot embed smoothly in $\R^4$. Since $X_0(K) \cong X_0(11n{34})$, this in turn implies that the Conway knot is not slice either. 
	\end{remark}

	\begin{exercise}As a toy example of Piccirillo's strategy, work out the following  alternative (and overkill!) proof that the knot $9_{42}$ is not slice (cf. Exercise~\ref{exer:942}): Use software to calculate the $s$-invariants of the knots in Figure~\ref{fig:not942}.\footnote{For example, use SnapPy \cite{snappy} to find a planar diagram or Dowker-Thistlethwaite code for each knot, then calculate  $s$-invariants using  KnotJob \cite{knotjob} or the {KnotTheory'} package \cite{knottheory} in Mathematica \cite{mathematica}.} You should find that the first knot $K=9_{42}$ has $s(K)=0$, whereas the second knot $K'$ has $s(K') \neq 0$. Apply the Trace Embedding Lemma to conclude that neither knot is slice.
	\end{exercise}
	
		\bigskip

	\section*{\textbf{Lecture 4: The contact and Stein setting}} 
	\addtocounter{section}{1}
	
		\setcounter{subsection}{0}
	\setcounter{theorem}{0}
		
		\medskip

	Next we will consider the geometric setting of Stein 4-manifolds and contact 3-manifolds arising as their boundaries.
	
	\subsection{Contact geometry and the slice-Bennequin inequality}
	
	Recall that a \emph{contact structure} on a 3-manifold $Y$ is a 2-plane distribution $\xi \subset TY$ that can be locally expressed as the kernel of a 1-form $\alpha$ such that $\alpha \wedge d\alpha \neq 0$; we call $\alpha$ a \emph{contact form}. 	
	\begin{example}
		The standard contact structure on $\R^3$ is given by $\xi=\ker(dz+x\, dy)$. This has  a natural extension to a standard contact structure on $S^3$; the latter   can also be realized as the set of affine complex lines that are tangent to the unit $S^3 \subset \C^2$.
	\end{example}
	
	A knot $K$ in a contact 3-manifold $(Y,\xi)$ is called \textit{Legendrian}, if it is tangent to $\xi$. Diagrammatically, Legendrian knots in the standard contact $\R^3$  are typically studied via the ``front projection'' $(x,y,z)\mapsto(y,z)$. 
	
	\begin{example}
		The front projection for a Legendrian representative of the (right-handed) trefoil is shown in Figure~\ref{fig:legtref}.
	\end{example}
	
	\begin{figure}[h]
		\centering
		\def\svgwidth{0.275\linewidth} 
\begingroup%
  \makeatletter%
  \providecommand\color[2][]{%
    \errmessage{(Inkscape) Color is used for the text in Inkscape, but the package 'color.sty' is not loaded}%
    \renewcommand\color[2][]{}%
  }%
  \providecommand\transparent[1]{%
    \errmessage{(Inkscape) Transparency is used (non-zero) for the text in Inkscape, but the package 'transparent.sty' is not loaded}%
    \renewcommand\transparent[1]{}%
  }%
  \providecommand\rotatebox[2]{#2}%
  \newcommand*\fsize{\dimexpr\f@size pt\relax}%
  \newcommand*\lineheight[1]{\fontsize{\fsize}{#1\fsize}\selectfont}%
  \ifx\svgwidth\undefined%
    \setlength{\unitlength}{279.36789926bp}%
    \ifx\svgscale\undefined%
      \relax%
    \else%
      \setlength{\unitlength}{\unitlength * \real{\svgscale}}%
    \fi%
  \else%
    \setlength{\unitlength}{\svgwidth}%
  \fi%
  \global\let\svgwidth\undefined%
  \global\let\svgscale\undefined%
  \makeatother%
  \begin{picture}(1,0.69065968)%
    \lineheight{1}%
    \setlength\tabcolsep{0pt}%
    \put(0,0){\includegraphics[width=\unitlength,page=1]{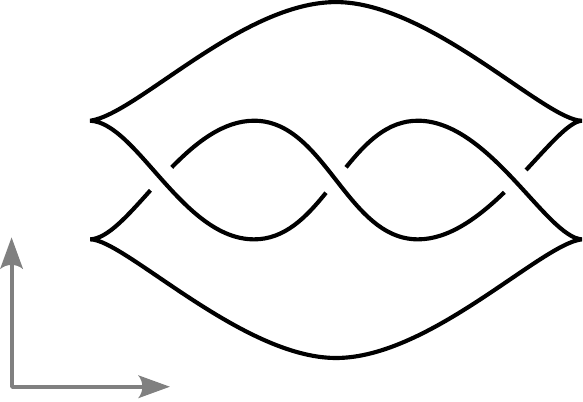}}%
    \put(-0.00087675,0.31464895){\color[rgb]{0.50196078,0.50196078,0.50196078}\makebox(0,0)[lt]{\smash{\begin{tabular}[t]{l}$z$\end{tabular}}}}%
    \put(0.30977696,0.01224977){\color[rgb]{0.50196078,0.50196078,0.50196078}\makebox(0,0)[lt]{\smash{\begin{tabular}[t]{l}$y$\end{tabular}}}}%
  \end{picture}%
\endgroup%

		\caption{A Legendrian representative of the right-handed trefoil.}
		\label{fig:legtref}
	\end{figure}

	The differential equation $x=-dz/dy$ imposes two basic features on front projection: cusps instead of vertical tangencies and overcrossing conditions. To state them more precisely, suppose a Legendrian knot is parametrized via $t \mapsto \sigma(t)=\left(x(t),y(t),z(t)\right)$ where  $\sigma'(t)$ is nowhere-vanishing: 
	
	\begin{enumerate}
		\item Front projections have cusps instead of vertical tangencies. Such points occur when $y'(t)=0$, which forces $z'(t)=0$, in turn implying that $\sigma'(t)$ points in the $\partial/\partial x$ direction whenever $y'(t)=0$.
		\item At crossings (which are required to project to a pair of transverse arcs), the overstrand has more negative slope. This is because the slope $dz/dy$ corresponds to the negative of the $x$-coordinate.
	\end{enumerate}
	It turns out that essentially any planar diagram satisfying these conditions is realized as the projection of a Legendrian knot or link (unique up to Legendrian isotopy). In a related vein, one can prove:
	
	\begin{proposition}
		Every knot can be made Legendrian by a small isotopy.
	\end{proposition}
	\begin{example}\label{figure8}
		We illustrate this for the figure-eight knot, also denoted $4_1$, in Figure~\ref{fig:leg-41}. Vertical tangencies are replaced with cusps. Two of the original crossings satisfy the overstrand condition, and we add zig-zags to fix the remaining two crossings.

		\begin{figure}[h!]
			\centering
			\includegraphics[width=0.49\linewidth]{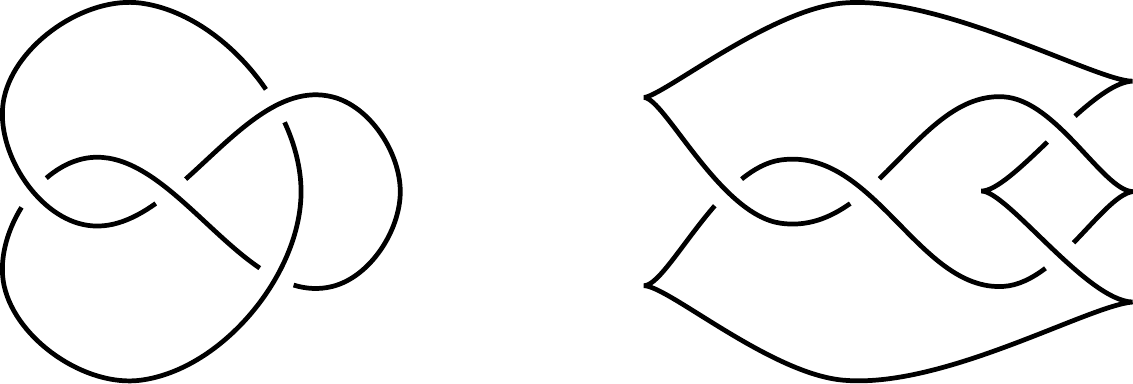}
			\caption{Bringing the figure-eight knot  into Legendrian position.}
			\label{fig:leg-41}
		\end{figure}
		
	\end{example}
		\newpage

	Legendrian knots have two ``classical'' invariants.
	
	\begin{definition}
		Given an oriented Legendrian knot diagram $D$, define
		\begin{itemize}
			\item the \emph{Thurston-Bennequin number} $\tb(D):=\writhe(D)-\frac{1}{2}\#\{\text{cusps}\}$, where the writhe is the difference between the numbers of positive and negative crossings; 
			\item the \emph{rotation number} $\rot(D):=\frac{1}{2}\left(\#\text{\{down cusps\}}-\#\text{\{up cusps\}}\right)$.	\end{itemize}
	\end{definition}
	
	These two quantities have  intrinsic geometric definitions as well, which is convenient for establishing their invariance under Legendrian isotopy (without appealing to a Legendrian version of Reidemeister's theorem). Roughly speaking, the rotation number measures the rotation of the tangent vector to the knot inside the contact planes relative to $\partial/\partial x$, and the Thurston-Bennequin number measures the linking between the knot and a pushoff along a vector field which is transverse to the contact planes.

	\begin{exercise}
		Fix an orientation on the Legendrian trefoil $K$ in Figure~\ref{fig:legtref} and verify that it has $\tb=1$ and $\rot =0$.
	\end{exercise}
	
	\begin{exercise}
		\begin{enumerate}
			\item[\textbf{(a)}] For the various possible orientations, find the Thurston-Bennequin and rotation numbers of the Legendrian unknots in Figure~\ref{fig:unknots}. 
			\item[\textbf{(b)}] In general, how does a change of orientation affect these invariants? 
			\item[\textbf{(c)}] How does the introduction of zig-zags (called \emph{stabilization}) affect these invariants?
		\end{enumerate}
		\begin{figure}
			\centering
			\includegraphics[width=0.4\linewidth]{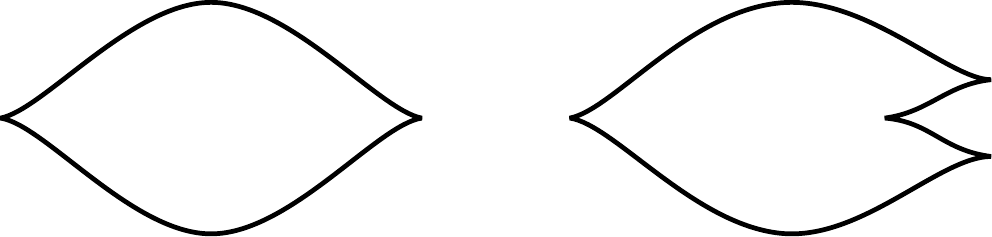}
			\caption{Two diagrams for the unknot with different Thurston-Bennequin and rotation numbers.}
			\label{fig:unknots}
		\end{figure}
	\end{exercise}
	
	These classical invariants can be used to articulate a constraint on the (smooth) slice genus of a knot; the strongest possible constraint is then obtained by minimizing over all Legendrian representatives of a given knot.
	
	\begin{theorem}[slice-Bennequin inequality]\label{thm:slice-Bennequin}
		Given a Legendrian representative of $K$, we have 
		$$2g_4(K)-1\geq \tb(K)+|\rot(K)|,$$
		where $g_4$ denotes the smooth slice genus.
	\end{theorem}
	
	The slice-Bennequin inequality was originally established by Rudolph \cite{rudolph:qp-obstruction} using the Thom conjecture (Theorem~\ref{thm:thom}) \cite{k-m:thom}. We will hint at an alternative proof  further below using the  ``adjunction inequality'' for surfaces in Stein domains.
	
	\begin{example}
		For each orientation, the Legendrian figure-eight knot in Figure~\ref{fig:leg-41} has 
		\begin{align*}
			\tb&= \writhe - \tfrac{1}{2} \# \{\text{cusps}\}=(2-2)-\tfrac{1}{2}(6)=-3
			\\
			\rot & = \frac{1}{2} \left(\#\text{\{down cusps\}}-\#\text{\{up cusps\}}\right) = \frac{1}{2}(3-3)=0.
		\end{align*}
		
		The slice-Bennequin inequality then implies   $2g_4-1\geq -3+|0|$, which reduces to $g_4\geq -1$. In this case, the slice-Bennequin inequality fails to provide a useful obstruction. Fortunately, there are many examples in which the bound is effective and superior to classical tools, such as the slicing obstructions from the Alexander polynomial.
	\end{example}

	\begin{exercise}
		Find a Legendrian representative of the pretzel knot $P(3,-5,-7)$ that has $\tb=1$ and $\rot=0$. \emph{(One solution is given in Figure~\ref{fig:leg-pretzel}.)} Then prove $P(3,-5,-7)$ is not (smoothly) slice. 
	\end{exercise}

	\begin{figure}
		\centering
		\includegraphics[width=0.42\linewidth]{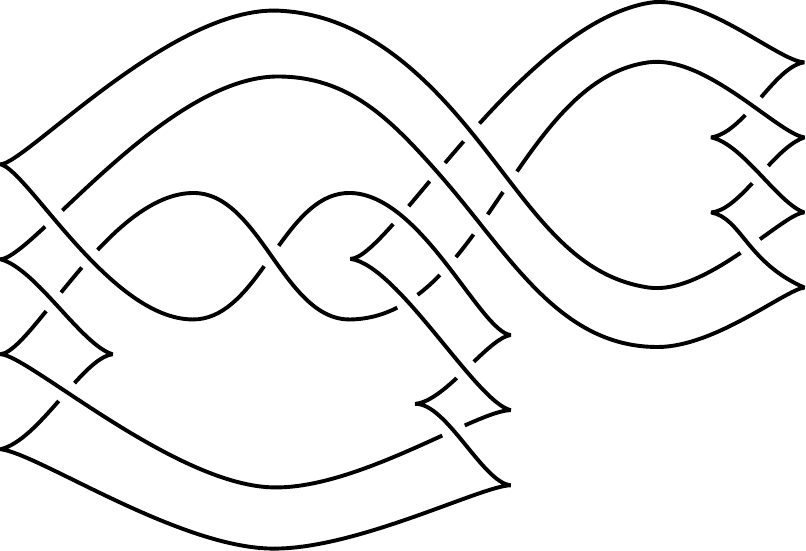}
		\caption{A Legendrian representative of the pretzel knot $P(3,-5,-7)$.}
		\label{fig:leg-pretzel}
	\end{figure}

	\subsection{Stein surfaces}
	Recall Whitney's embedding theorem, which states that any smooth $n$-manifold can be smoothly embedded into $\R^{2n}$. In contrast, many complex manifolds do not admit  holomorphic embeddings into $\C^n$ for any $n$. (For example, embeddings of closed complex manifolds $Z$ are obstructed by the maximum principle for holomorphic maps, since the coordinate projections $Z \hookrightarrow \C^n \to \C$ must be constant.) The class of complex manifolds  that do admit embeddings into some $\C^n$ are called \textit{Stein manifolds}.
	
	\begin{definition}
		A \textit{Stein surface} $Z$ is a complex surface ($\dim_\C Z=2$) admitting a proper holomorphic embedding into $\C^n$ for some $n$.
	\end{definition}
	
	\begin{remark}
		This is not the original definition (which involves the notions of holomorphic convexity and separability), but it is an elegant alternative  characterization.
	\end{remark}
	
	A natural compact analogue is a \textit{Stein domain}. For our purposes, we will view these as the compact 4-manifolds $X$ (with boundary) arising as subsets $Z \cap B^{2n}(r)$ where $Z \subset \C^n$ is a Stein surface and $B^{2n}(r)$ is a ball of radius $r>0$.  The boundary inherits a natural contact structure by considering the complex tangencies. (In particular, at each point $x \in \partial X$, the contact plane $\xi_x$ is the unique plane in the 3-dimensional tangent space $T_x \partial X$ that is fixed under complex multiplication in the complex vector space  $T_x \C^n$.) We also note that manifolds of this form have a natural symplectic structure.
	
	The following important property of Stein domains brings us back to the minimal genus problem considered in the previous lecture.
	
	\begin{theorem}[Adjunction inequality \cite{lisca-matic}]\label{thm:adjunction}
		If $X$ is a Stein domain and $\Sigma\subset X$ is a smoothly embedded surface representing a nonzero class in $H_2(X)$, then
		$$[\Sigma] \cdot [\Sigma]+\big|\langle c_1(X,[\Sigma])\rangle \big|\leq 2g(\Sigma)-2,$$
		where $c_1(X)$ denotes the first Chern class of $X$ as a complex manifold.
	\end{theorem}

	How do we calculate this in practice? Work of Eliashberg \cite{yasha:stein,yasha:knots} and Gompf \cite{gompf:stein} establishes  a handle theory for Stein domains that admits a convenient diagrammatic framework. For our purposes, we focus on the case of 4-manifolds built with a single 0-handle and some number of 2-handles.
	
	\begin{theorem}[Eliashberg \cite{yasha:stein,yasha:knots}, cf. Weinstein \cite{weinstein}]\label{thm:stein}
		Let $X$ be a 4-manifold obtained by attaching 2-handles to $B^4$ along a framed, oriented Legendrian link $L=\sqcup_1^n K_i$  in the standard contact $S^3$. If the framing on each component $K_i$ equals $\tb(K_i)-1$, then $X$ admits a Stein structure. 
		
		Moreover, the first Chern class $c_1(X) \in H^2(X)$ is determined by the rotation numbers of $L$ via the formula $$\langle c_1(X),h_i\rangle =\rot(K_i),$$
		where $h_i$ is the generator of $H_2(X)$ corresponding to the 2-handle attached along $K_i$ (with orientation).
	\end{theorem}
	
	Note that reversing the orientation of $K_i$ has the effect of replacing $h_i$ with $-h_i$.
	
	\begin{example} Since the Legendrian trefoil $K$ in Figure~\ref{fig:legtref} has $\tb(K)=1$, its 0-framed knot trace $X=X_0(K)$ admits a Stein structure. Moreover, since  $\rot(K)=0$, we see that the class $c_1(X) \in H^2(X) \cong \Z$ vanishes. Applying the adjunction inequality, we see that any  smoothly embedded surface $\Sigma \subset X$ representing a nontrivial homology class satisfies $g(\Sigma) \geq 1$ (using the fact that $[\Sigma]\cdot [\Sigma]=0$).
	\end{example}

	\begin{exercise}
		Given a Legendrian knot $K$ in the standard contact $S^3$, use Theorems~\ref{thm:stein} and the adjunction inequality (Theorem~\ref{thm:adjunction}) to establish the slice-Bennequin inequality for $K$ (Theorem~\ref{thm:slice-Bennequin}).
	\end{exercise}
	
	By similar methods we can get constraints on the genera of surfaces in many other 4-manifolds as well (cf.~\cite{akbulut-matveyev}).
	
	{\small
	\bibliographystyle{alpha}
	\bibliography{biblio}}

\newcommand{\etalchar}[1]{$^{#1}$}
\begin{thebibliography}{FdBMNS24}

\bibitem[ACM{\etalchar{+}}23]{acmps}
Paolo Aceto, Nickolas~A Castro, Maggie Miller, JungHwan Park, and Andr{\'a}s
  Stipsicz.
\newblock Slice obstructions from genus bounds in definite 4-manifolds.
\newblock {\em arXiv preprint arXiv:2303.10587}, 2023.

\bibitem[AM97]{akbulut-matveyev}
S.~Akbulut and R.~Matveyev.
\newblock Exotic structures and adjunction inequality.
\newblock {\em Turkish J. Math.}, 21(1):47--53, 1997.

\bibitem[CDGW]{snappy}
Marc Culler, Nathan~M. Dunfield, Matthias Goerner, and Jeffrey~R. Weeks.
\newblock Snap{P}y, a computer program for studying the geometry and topology
  of $3$-manifolds.
\newblock \url{http://snappy.computop.org}.

\bibitem[DKM{\etalchar{+}}24]{dkmps}
Irving Dai, Sungkyung Kang, Abhishek Mallick, JungHwan Park, and Matthew
  Stoffregen.
\newblock The (2,1)-cable of the figure-eight knot is not smoothly slice.
\newblock {\em Inventiones mathematicae}, 238(2):371--390, 2024.

\bibitem[Don83]{donaldson}
S.~K. Donaldson.
\newblock An application of gauge theory to four-dimensional topology.
\newblock {\em J. Differential Geom.}, 18(2):279--315, 1983.

\bibitem[Eli90]{yasha:stein}
Ya. Eliashberg.
\newblock Topological characterization of {S}tein manifolds of dimension
  {$>2$}.
\newblock {\em Internat. J. Math.}, 1(1):29--46, 1990.

\bibitem[Eli93]{yasha:knots}
Ya. Eliashberg.
\newblock {L}egendrian and transversal knots in tight contact $3$-manifolds.
\newblock In {\em Topological methods in modern mathematics (Stony Brook, NY,
  1991)}, pages 171--193. Publish or Perish, Houston, TX, 1993.

\bibitem[FdBMNS24]{singularities:book}
Javier Fern{\'a}ndez~de Bobadilla, Marco Marengon, Andr{\'a}s N{\'e}methi, and
  Andr{\'a}s Stipsicz, editors.
\newblock {\em Singularities and low dimensional topology. {Lecture} notes of
  the spring semester 2023, {Erd{\H{o}}s} {Center}, {Budapest}, {Hungary}
  2023}, volume~30 of {\em Bolyai Soc. Math. Stud.}
\newblock Cham: Springer, 2024.

\bibitem[FM66]{fox-milnor}
Ralph~H. Fox and John~W. Milnor.
\newblock Singularities of {$2$}-spheres in {$4$}-space and cobordism of knots.
\newblock {\em Osaka Math. J.}, 3:257--267, 1966.

\bibitem[FQ90]{freedman-quinn}
Michael~H. Freedman and Frank Quinn.
\newblock {\em Topology of 4-manifolds}, volume~39 of {\em Princeton
  Mathematical Series}.
\newblock Princeton University Press, Princeton, NJ, 1990.

\bibitem[Fre82]{freedman}
Michael~Hartley Freedman.
\newblock The topology of four-dimensional manifolds.
\newblock {\em J. Differential Geom.}, 17(3):357--453, 1982.

\bibitem[Gom98]{gompf:stein}
R.~Gompf.
\newblock Handlebody construction of {S}tein surfaces.
\newblock {\em Ann. of Math. (2)}, 148(2):619--693, 1998.

\bibitem[GS99]{gompfstipsicz}
Robert~E. Gompf and Andr\'{a}s~I. Stipsicz.
\newblock {\em {$4$}-manifolds and {K}irby calculus}, volume~20 of {\em
  Graduate Studies in Mathematics}.
\newblock Amer.~Math.~Society, Providence, RI, 1999.

\bibitem[Kho00]{khovanov00}
Mikhail Khovanov.
\newblock A categorification of the {J}ones polynomial.
\newblock {\em Duke Math. J.}, 101(3):359--426, 2000.

\bibitem[KM61]{kervaire-milnor}
Michel~A. Kervaire and John~W. Milnor.
\newblock On 2-spheres in 4-manifolds.
\newblock {\em Proceedings of the National Academy of Sciences of the United
  States of America}, 47(10):1651--1657, 1961.

\bibitem[KM94]{k-m:thom}
P.~B. Kronheimer and T.~S. Mrowka.
\newblock The genus of embedded surfaces in the projective plane.
\newblock {\em Mathematical Research Letters}, Volume 1(6):797 -- 808, 1994.

\bibitem[kno]{knottheory}
\texttt{KnotTheory`}.
\newblock Mathematica Package, \texttt{katlas.math.toronto.edu}, accessed June
  2013.

\bibitem[LM98]{lisca-matic}
P.~Lisca and G.~Mati{\'c}.
\newblock Stein $4$-manifolds with boundary and contact structures.
\newblock {\em Topology Appl.}, 88:55--66, 1998.

\bibitem[LP72]{LP72}
Fran\c{c}ois Laudenbach and Valentin Po\'{e}naru.
\newblock A note on {$4$}-dimensional handlebodies.
\newblock {\em Bull. Soc. Math. France}, 100:337--344, 1972.

\bibitem[Man16]{manolescu}
Ciprian Manolescu.
\newblock {P}in(2)-equivariant {S}eiberg-{W}itten {F}loer homology and the
  triangulation conjecture.
\newblock {\em Journal of the American Mathematical Society}, 29(1):147--176,
  2016.

\bibitem[OS00]{os:symplectic}
Peter Ozsvath and Zoltan Szabo.
\newblock The symplectic {T}hom conjecture.
\newblock {\em Annals of Mathematics}, 151(1):93--124, 2000.

\bibitem[Pic20]{picc:conway}
Lisa Piccirillo.
\newblock The {C}onway knot is not slice.
\newblock {\em Ann. of Math. (2)}, 191(2):581--591, 2020.

\bibitem[Qui82]{quinn:ends}
Frank Quinn.
\newblock {Ends of maps. III. Dimensions 4 and 5}.
\newblock {\em Journal of Differential Geometry}, 17(3):503 -- 521, 1982.

\bibitem[Ras10]{rasmussen:s}
Jacob Rasmussen.
\newblock Khovanov homology and the slice genus.
\newblock {\em Invent. Math.}, 182(2):419--447, 2010.

\bibitem[Roh52]{rohlin}
Vladimir Rohlin.
\newblock New results in the theory of four-dimensional manifolds.
\newblock {\em Doklady Acad. Nauk. SSSR (N.S.)}, 84:221--224, 1952.

\bibitem[Rud83]{rudolph:braided-surface}
L.~Rudolph.
\newblock Braided surfaces and {S}eifert ribbons for closed braids.
\newblock {\em Comment. Math. Helv.}, 58(1):1--37, 1983.

\bibitem[Rud90]{rudolph:kauffman-bound}
L.~Rudolph.
\newblock A congruence between link polynomials.
\newblock {\em Math. Proc. Cambridge Philos. Soc.}, 107(2):319--327, 1990.

\bibitem[Rud93]{rudolph:qp-obstruction}
L.~Rudolph.
\newblock Quasipositivity as an obstruction to sliceness.
\newblock {\em Bull. Amer. Math. Soc. (N.S.)}, 29(1):51--59, 1993.

\bibitem[Sav12]{saveliev}
Nikolai Saveliev.
\newblock {\em An Introduction to the Casson Invariant}.
\newblock De Gruyter, Berlin, Boston, 2012.

\bibitem[Sch21]{knotjob}
Dirk Sch\"{u}tz.
\newblock {K}not{J}ob.
\newblock Available at \url{http://www.maths.dur.ac.uk/~dma0ds/knotjob.html},
  2021.

\bibitem[Wei91]{weinstein}
Alan Weinstein.
\newblock Contact surgery and symplectic handlebodies.
\newblock {\em Hokkaido Math. J.}, 20(2):241--251, 1991.

\bibitem[{Wol}]{mathematica}
{Wolfram Research{,} Inc.}
\newblock Mathematica, {V}ersion 13.2.
\newblock Champaign, IL, 2022.

\end{thebibliography}
	
\end{document}